\documentclass[12pt]{amsart}

\usepackage[hmargin=2cm,bottom=3cm,textwidth=17.3cm]{geometry}

\usepackage{youngtab}
\usepackage[latin1]{inputenc}
\usepackage[all]{xy}
\usepackage{mathrsfs}
\usepackage{amsmath}
\usepackage{amssymb}
\usepackage{amsthm}
\usepackage{enumerate}
\usepackage{pdfsync}
\usepackage[percent]{overpic}
\usepackage[linktocpage=true]{hyperref}
\usepackage{fancyhdr}
\usepackage{tikz}
\usetikzlibrary{decorations.pathreplacing}

\newcommand{\op}{\operatorname}
\newcommand{\Ext}{\operatorname{Ext}}

\numberwithin{equation}{section}

\newcommand{\stb}{,\ldots,}

\newcommand{\Ann}{\operatorname{Ann}}

\newcommand{\coker}{\operatorname{coker}}

\newcommand{\iso}{\cong}

\newcommand{\im}{\operatorname{Im}}

\newcommand{\bra}{\langle}
\newcommand{\ket}{\rangle}

\newcommand{\Sq}{\operatorname{Sq}}

\newcommand{\RP}{\mathbb{R}P}

\newcommand{\Gr}{\operatorname{Gr}}
\newcommand{\OGr}{\widetilde{\operatorname{Gr}}}

\newcommand{\al}{\alpha}
\newcommand{\be}{\beta}

\newcommand{\de}{\delta}

\newcommand{\la}{\lambda}

\newcommand{\om}{\omega}

\newcommand{\se}{\subseteq}

\newcommand{\R}{\mathbb{R}}

\newcommand{\F}{\mathbb{F}}

\newtheorem{fact}{Fact}[section]
\newtheorem{lemma}[fact]{Lemma}
\newtheorem{theorem}[fact]{Theorem}
\newtheorem{defi}[fact]{Definition}
\newtheorem{exa}[fact]{Example}
\newtheorem{cla}[fact]{Claim}
\newtheorem*{sol}{\it Solution}
\newtheorem{proposition}[fact]{Proposition}
\newtheorem{corollary}[fact]{Corollary}
\newtheorem{conjecture}[fact]{Conjecture}

\newenvironment{definition}{\begin{defi} \rm}{\end{defi}}
\newenvironment{example}{\begin{exa} \rm}{\end{exa}}

\theoremstyle{remark}
\newtheorem{nremark}[fact]{Remark}
\newenvironment{remark}
{\par\pushQED{\qed}\nremark \small}
{\popQED\endnremark}

\newcommand{\crk}{\mathrm{crk}}
\Yboxdim{5pt}

\begin{document}

\title[Mod 2 cohomology of oriented Grassmannians]{The mod 2 cohomology rings of oriented Grassmannians via Koszul complexes}

\address{\'Akos K.\ Matszangosz, HUN-REN Alfr\'ed R\'enyi Institute of Mathematics, Re\'altanoda utca 13-15, 1053 Budapest, Hungary}
\email{matszangosz.akos@gmail.com}

\address{Matthias Wendt, Fachgruppe Mathematik und Informatik, Bergische Universit\"at Wuppertal, Gaussstrasse 20, 42119 Wuppertal, Germany}
\email{m.wendt.c@gmail.com}

\thanks{\'A. K. M. is supported by the Hungarian National Research, Development and Innovation Office, NKFIH K 138828.}


\begin{abstract}
  We study the structure of mod 2 cohomology rings of oriented Grassmannians $\OGr_k(n)$ of oriented $k$-planes in $\R^n$. Our main focus is on the structure of the cohomology ring ${\rm H}^*(\OGr_k(n);\F_2)$ as a module over the characteristic subring $C$, which is the subring generated by the Stiefel--Whitney classes $w_2\stb w_k$. We identify this module structure using Koszul complexes, which involves the syzygies between the relations defining $C$. We give an infinite family of such syzygies, which results in a new upper bound on the characteristic rank of $\OGr_k(2^t)$, and formulate a conjecture on the exact value of the characteristic rank of $\OGr_k(n)$. For the case $k=3$, we use the Koszul complex to compute a presentation of the cohomology ring $H={\rm H}^*(\OGr_3(n);\F_2)$ for $2^{t-1}<n\leq 2^t-4$, complementing existing descriptions in the $n=2^t-3,...,2^t$ cases. 
  More precisely, as a $C$-module, $H$ splits as a direct sum of the characteristic subring $C$ and the anomalous module $H/C$, and we compute a complete presentation of $H/C$ as a $C$-module from the Koszul complex. We also discuss various issues that arise for the cases $k>3$, supported by computer calculation.
\end{abstract}

\author{\'Akos K.\ Matszangosz and Matthias Wendt}
\maketitle

\setcounter{tocdepth}{1}
\tableofcontents

\section{Introduction}

The cohomology of real Grassmannians is by now fairly well understood, both with mod~2 and with integral coefficients, by ways of Schubert calculus. It might then be very surprising that something as innocuous as taking a double cover can produce something as little understood as the oriented Grassmannians. While the mod~2 and rational Betti numbers are known, the mod~2 cohomology ring structure as well as the integral cohomology are still fairly mysterious. In this paper, we study the mod~2 cohomology ring of the oriented Grassmannians $\OGr_3(n)$ of oriented 3-planes in an oriented real $n$-dimensional vector space, using Koszul complexes. 

\subsection{Question and known results}
To give an idea of the context and known results pertaining to the cohomology ring structure for the oriented Grassmannians, we give a brief overview before formulating our results in the next section.

Before even getting to the ring structure, recall that the mod~2 Betti numbers for the oriented Grassmannians are known by work of Ozawa \cite{Ozawa2021}, using Morse theory.

Turning to the ring structure, there are some general approaches that can be used for homogeneous spaces, e.g. the Eilenberg--Moore spectral sequence for the fibration $G/H\to {\rm B}H\to {\rm B}G$, combined with information on the cohomology rings of classifying spaces. Many results have been achieved with this technique, see e.g. the papers of Borel \cite{Borel1953_mod2}, Baum \cite{Baum1968} and Franz \cite{Franz2021}. Such techniques work very well for the additive structure. For the multiplicative structure, there are extension problems, but in \cite{Franz2021}  these extension problems are solved under the assumption that 2 is invertible in the coefficients. The case of mod~2 coefficients seems to be the most difficult, as examples are known where the passage from the $E_\infty$-page to the actual cohomology involves nontrivial extensions.

This means that some additional information is necessary to investigate the ring structure in the oriented Grassmannian case. In \cite{KorbasRusin2016:cohomology}, Korba\v s and Rusin determined the ring structure for $\OGr_2(n)$ using the Gysin sequence combined with information on the characteristic rank and the image of the pullback along the double cover from \cite{KorbasRusin2016:rank}. 

Beyond the case $\OGr_2(n)$, no complete information on cohomology rings is available. In the cases $\OGr_k(n)$, $k=3$ or $4$, information on the characteristic rank and the image of the pullback has been obtained using Gr\"obner bases in \cite{Rusin2017}, \cite{PetrovicPrvulovicRadovanovic} (for $k=3$) and \cite{Korbas2015}, \cite{PrvulovicRadovanovic2019} (for $k=4$). Computations of cohomology rings have been made in some cases: Basu and Chakraborty \cite{BasuChakraborty2020} partially computed the cohomology ring structures for $\OGr_k(n)$ with $k=3$ and $n=2^{t-3},...,2^t$, and remaining ambiguities in the relations were recently resolved by Colovi\'c--Prvulovi\'c \cite{ColovicPrvulovic2023_2t} and Jovanovi\'c--Prvulovi\'c \cite{JovanovicPrvulovic2023}. Our focus in this paper is the description of the remaining cases for $k=3$, when $2^{t-1}<n\leq 2^t-4$.

For a few more known computations, Jovanovi\'c also recently determined integral cohomology for some $\OGr_3(n)$ cases in \cite{Jovanovic2022}, and Rusin computed some mod~2 cohomology rings for $\OGr_4(n)$ in~\cite{Rusin2019}. 

\subsection{The general setup}
After the overview of the literature, we now turn to describe a simple framework for computing the mod~2 cohomology of $\OGr_k(n)$. Before we can formulate our results, we need to set up some background and notation, which is discussed in more detail in Section~\ref{sec:general} below. The first thing to note is that the Gysin sequence for the double cover $\OGr_k(n)\to \Gr_k(n)$ produces an exact sequence of $C$-modules
\begin{equation}\label{eq:introgysin}
	\xymatrix@R-2pc{
		0\ar[r]&C\ar[r]&{\rm H}^*(\OGr_k(n),\mathbb{F}_2)\ar[r]^-{\de}&K\ar[r]&0.
	}	
\end{equation}
where $C$ and $K$ are the cokernel and kernel of the map $w_1\colon {\rm H}^*(\Gr_k(n),\mathbb{F}_2)\to {\rm H}^*(\Gr_k(n),\mathbb{F}_2)$ given by multiplication with the first Stiefel--Whitney class of the tautological subbundle. 
This immediately poses three questions: 
\begin{itemize}
	\item[a)]describe $C$ using generators and relations,
	\item[b)]describe $K$ as a $C$-module using generators and relations, and
	\item[c)]determine the extension class in $\Ext_C(K,C)$ given by the exact sequence \eqref{eq:introgysin}.
\end{itemize}

For step a), the ring $C$ has a well-known \cite{Fukaya} explicit description as follows:
\begin{equation}\label{eq:Cintro}
	C=\F_2[w_2\stb w_k]/(q_{n-k+1}\stb q_n),	
\end{equation}
with $w_i=w_i(S)$ and $q_i=w_i(\ominus S)$ where $\ominus S$ is the formal inverse of the tautological bundle $S\to {\rm BSO}(k)$. The $q_i$ can be expressed in terms of $w_i$ via a Giambelli type formula \eqref{eq:giambelli}, \eqref{eq:qjfromQj} or a recursion, see \ref{eq:recursion}. A lot of information on the structure of $C$ as a commutative ring can be obtained via Gr\"obner basis methods used in many papers, e.g. \cite{Fukaya,PetrovicPrvulovicRadovanovic,PrvulovicRadovanovic2019,ColovicPrvulovic2023_2t}.
 
The next step b) -- and the one this paper is really focused on -- is to compute a presentation for $K$ as a $C$-module. To determine $C$ and $K$, we note that they appear as 0th and 1st Koszul homology groups for the ideal $I=(q_{n-k+1},\dots,q_n)$ over the ring $W_2=\mathbb{F}_2[w_2,\dots,w_k]$. The generators of $K$ are directly related to the syzyzgies between generators of the ideal $I$ in $W_2$, see the discussion in Section~\ref{sec:koszul-homology}. 

After determining the extension class in step c), most of the information relevant for the mod~2 cohomology ring is available, namely all products where one factor is in $C$. We settle this description for $k=3$, in the cases that are not covered in the literature. 
\subsection{The case $k=3$}

 Next, we summarize our computations of the mod~2 cohomology rings of $\OGr_3(n)$. As detailed above, the cases $n=2^t-3,\dots,2^t$ have been mostly settled in the literature, and we will focus on the remaining cases $2^{t-1}<n\leq 2^t-4$ in the present paper. 
 
The reason for the distinction between the two cases $n=2^t-3,\dots,2^t$ and $2^{t-1}<n\leq 2^t-4$ is that in the former case, the one dealt with in the literature, the $C$-module $K$ is free of rank one, so the extension automatically splits, answering both questions b) and c). In the cases $2^{t-1}<n\leq 2^t-4$ that we are dealing with here, $K$ is no longer free, and it is generated by two elements $a_n$ and $d_n$. 
The degrees of the generators were already determined in \cite{BasuChakraborty2020}, but the relations in the presentation of $K$ have not been determined before. Our key computation in Section \ref{sec:K-presentation} is based on a detailed investigation of the Koszul complex for the ideal $I=(q_{n-2},q_{n-1},q_n)$ over the ring $W_2=\F_2[w_2,w_3]$. The main results of the paper give a presentation for $K$ as a $C$-module, answering question b).
Having that, it turns out that the above exact sequence splits as an extension of $C$-modules for degree reasons, see Proposition~\ref{prop:ext-vanishing-k3}, which settles the extension problem. In conclusion, we obtain the following description of the cohomology ring of $\OGr_3(n)$ for $2^{t-1}<n\leq 2^t-4$.

\break
\begin{theorem}\label{thm:main}
  Let $2^{t-1}<n\leq 2^t-4$, let $C$ be defined as in \eqref{eq:Cintro} and set $i=2^t-3-n$ and $j=n-2^{t-1}+1$. Then we have the isomorphism of $C$-modules
  \begin{equation}
    {\rm H}^*(\OGr_3(n);\F_2)\iso C\bra 1,a_n,d_n\ket/(q_ia_n+r_{j-1}d_n,
    q_{i+1}a_n+w_3r_{j-2}d_n, 
    w_3q_{i-1}a_n+r_jd_n),
  \end{equation}
  where $\deg a_n=3n-2^t-1$ and $\deg d_n=2^t-4$. Here $q_i$ are polynomials in $\F_2[w_2,w_3]$ defined by the recursion $q_{i}=w_2q_{i-2}+w_3q_{i-3}$ with $q_0=1$, $q_{<0}=0$, and $r_j$ are polynomials in $\mathbb{F}_2[w_2,w_3]$ defined by the recursion
  \[
  r_{j+1}=w_2r_j+w_3^2r_{j-2}.
  \]
  with $r_0=1$, $r_{<0}=0$. Closed-form expressions for $q_j$ and $r_j$ can be found in \eqref{eq:multinomial} and \eqref{eq:rj_closed}, respectively. The remaining ring structure of ${\rm H}^*(\OGr_3(n),\F_2)$ is determined by
  \begin{equation}\label{eq:squares}
    a_n^2=a_nd_n=d_n^2=0
  \end{equation}
\end{theorem}

The proof proceeds in the following steps:
\begin{itemize}
	\item[(I)] In Proposition~\ref{prop:k3-basis} we show that the kernel of the differential in the Koszul complex is a free $W_2$-module on two elements:
	\[\ker\left( d_1\colon W_2^{\oplus 3}\to W_2\right)=W_2\bra u_{3n-2^t},v_{2^t-3}\ket, \]
	\item[(II)] we compute the relations in the presentation of $K$ as a $C$-module (using the Koszul differential $d_2$) in Proposition~\ref{prop:K-relations},
	\item[(III)] we show that we have a splitting of $C$-modules ${\rm H}^*(\OGr_3(n),\mathbb{F}_2)=C\oplus K$ in  Proposition~\ref{prop:ext-vanishing-k3}, 
	\item[(IV)] and we compute the remaining products \eqref{eq:squares} in Proposition~\ref{prop:products}.
\end{itemize}

\subsection{The general method, and the case $k=4$}

The approach used for $k=3$ can also help computations and get some mileage in the cases $k>3$. It should be noted, however, that already the situation for $k=4$ differs in a significant number of aspects from the $k=3$ case, which really appears to be unusually smooth. Part of the difficulty of computing the cohomology rings ${\rm H}^*(\OGr_k(n);\F_2)$ for $k>3$ is related to the difficulty of computing a presentation of the first Koszul homology group as a $C$-module. We briefly outline the general method and mention some aspects that fail for $k>3$ here. For a more detailed discussion, see Section~\ref{sec:general} for the general method and Section~\ref{sec:beyondk3} for a description of the phenomena in the $k>3$ cases. 
\subsubsection{The Koszul complex}
The short exact sequence \eqref{eq:introgysin} derived from the Gysin sequence as well as the description of $C$ and $K$ as the zeroth and first homology of the Koszul complex works in complete generality. From this description, the generators of $K$ are directly related to the syzyzgies between generators of the ideal $I=(q_{n-k+1},\dots,q_n)$ in $W_2=\mathbb{F}_2[w_1,\dots,w_k]$, see the discussion in Section~\ref{sec:koszul-homology}.

We find some such syzygies in Theorem \ref{thm:twopower_syzygies}; namely, we prove that for $n=2^t$, the following relation holds between $(q_{n-k+1}\stb q_n)$:
\begin{theorem}
  \label{thm:fundamental-intro}
  For $n=2^t$, for $0<k<2^t$:
  \begin{equation}\label{eq:intro_relation}
    \sum_{i \text{ even}}q_{n-i}w_{i}=\sum_{i>1 \text{ odd}} q_{n-i}w_{i}=0.
  \end{equation}
\end{theorem}

The relation \eqref{eq:intro_relation} is a generalization of a result of Fukaya and Korba\v s about the vanishing of $q_{2^t-3}$ in the cases $k=3,4$, see \cite{Korbas2015}, \cite{Fukaya}, which is crucial in understanding the characteristic rank in these cases. The above result provides a $C$-module generator of the anomalous module $K$, and thus provides a new upper bound on the characteristic rank of $\OGr_k(2^t)$ for general $k$ and $t$, see Theorem \ref{thm:charrank}.

We also develop a general technique of ``ascending'' and ``descending'' such relations between the $q_j$ in Section~\ref{sec:syzygies}, which allows to obtain such relations for $\OGr_k(n)$ from similar relations for $\OGr_k(n-1)$ or $\OGr_k(n+1)$. In the case $k=3$, all syzygies (and thus all generators of $K$) are obtained from the vanishing $q_{2^t-3}=0$ via ascending and descending relations. This is already not true for $k=4$, as we demonstrate in Section \ref{sec:beyondk3}. However, the technique of ascending and descending relations provides syzygies for arbitrary $n$. Combining the fundamental relation in Theorem~\ref{thm:fundamental-intro} with the technique of ascending and descending relations, we formulate a general conjecture on the characteristic rank in Conjecture~\ref{new-amazing-conjecture} which is supported by computer experiments for small $k$ and $n$, see the discussion in Section~\ref{sec:beyondk3}.

An intermediate step in the computation of $K$ as a $C$-module --- viewed as the first homology of the Koszul complex --- is the computation of the kernel of the differential $d_1\colon \mathcal{K}_1\to\mathcal{K}_0$, cf.~Definition~\ref{def:koszul}. In the case $k=3$, this kernel is free of rank 2, generated exactly by the ascended and descended relations. This fails for $k=4$: more generators are necessary to generate the kernel, which is also no longer free. Computational experiments suggest that the kernel (as a $W_2$-module) always has a free resolution of length $k-2$. Nevertheless, it seems possible to obtain presentations for the kernel in the $k=4$ case, and once this is done, the relations in a presentation of $K$ as $C$-module can be extracted from the differential $d_2\colon\mathcal{K}_2\to \mathcal{K}_1$ much as in the present paper.
\subsubsection{The extension class}
Once a presentation of $K$ as a $C$-module has been obtained, we can ask how to determine the class of the extension $0\to C\to {\rm H}^*(\OGr_k(n),\mathbb{F}_2)\to K\to 0$ as an element in ${\rm Ext}^1_C(K,C)$. In the $k=3$ case, the Ext-group vanishes for degree reasons, see the discussion in Section~\ref{sec:ext-class}. However, computer algebra experiments show that the Ext-group is not generally trivial already for some $k=4$ cases. This non-vanishing of the Ext-group means that determining the $C$-module structure on ${\rm H}^*(\OGr_k(n),\mathbb{F}_2)$ is potentially much more complicated for $k>3$, a level of difficulty that seems to have not been noticed before. Nevertheless, one could imagine that the explicit description of the Koszul complex and kernel as in the previous step will help in computing the Ext-group more conceptually (at least in the $k=4$ case). This computation, and the question how to determine the class of the extension for ${\rm H}^*(\OGr_k(n),\mathbb{F}_2)$ in this Ext-group (which is equivalent to computing a presentation of the cohomology as a $C$-module), will be the focus of future research.

\subsubsection{Remaining ambiguities}
Once we understand ${\rm H}^*(\OGr_k(n),\mathbb{F}_2)$ as a $C$-module, most of the product structure is described. All that remains is to compute products between (lifts of) generators of $K$ as elements in ${\rm H}^*$. In the case $k=3$, there are only three products that need to be determined, one of which is trivial for degree reasons. The squares are shown to be zero via a similar induction as the one used for the relations, cf.~Propositions~\ref{prop:descendedsquare} and \ref{prop:ascendedsquare}. For $k>3$, one can well imagine to use integral cohomology information or cohomology operations to remove the remaining ambiguities as was done in recent papers \cite{ColovicPrvulovic2023_2t,JovanovicPrvulovic2023}. There are also some $k=4,5,6$ cases where the Ext-group vanishes and the remaining ambiguities can be removed easily, which we discuss in Section~\ref{sec:beyondk3}.

\subsection{Structure of the paper} We begin in Section~\ref{sec:general} with a review of possible techniques that have been employed to compute the cohomology of the oriented Grassmannians. We also emphasize the questions of extensions and module structures over the characteristic subring. Section~\ref{sec:characteristic-subring} provides information on the characteristic subring and its properties. In Section~\ref{sec:syzygies} we give a syzygy between the Stiefel--Whitney polynomials defining the characteristic subring and a new inductive procedure to obtain further relations. Then Section~\ref{sec:koszul-homology} recalls Koszul complexes and provides the concrete identification of the anomalous module $K$ as first Koszul homology. This allows to compute an explicit presentation of $K$ as a module over the characteristic subring in Section~\ref{sec:K-presentation}. A general discussion of Ext-groups and the vanishing result for the $k=3$ case is provided in Section~\ref{sec:ext-class}, and the remaining products of anomalous generators are investigated in Section~\ref{sec:ambiguities}. We also include an extensive discussion of the differences in the $k>3$ cases in Section~\ref{sec:beyondk3}.

\section{General method and notations}
\label{sec:general}

In this section we describe our general approach as well as introduce some notation and terminology. We also recall some other possible approaches and show how Poincar\'e duality completely settles the $C$-module question in cases when $K$ is generated by a single element.

\subsection{The Gysin sequence and setup of notation}

We denote by $\Gr_k(n)$ the real Grassmannian of $k$-planes in $\R^n$ and by $\OGr_k(n)$ the Grassmannian of oriented $k$-planes in $\R^n$. At the core of the upcoming algebraic computations is the Gysin sequence associated to the degree 2 covering $\pi\colon \OGr_k(n)\to\Gr_k(n)$: 
\[\xymatrix{	\ar[r]&
	{\rm H}^{i-1}(\Gr_k(n);\F_2)\ar[r]^{w_1}& 
	{\rm H}^i(\Gr_k(n);\F_2) \ar[r]^{\pi^*}& 
	{\rm H}^i(\OGr_k(n);\F_2) \ar[r]^\de& 
	{\rm H}^{i}(\Gr_k(n);\F_2)\ar[r]&
}
\]
In particular, the cohomology of $\OGr_k(n)$ sits in the short exact sequence
\begin{equation}\label{eq:w1ses}
	\xymatrix{	
		0\ar[r]&
		\coker w_1 \ar[r]^-{\pi^*}& 
		{\rm H}^*(\OGr_k(n);\F_2) \ar[r]^-\de& 
		\ker w_1\ar[r]&0
	}
\end{equation}
where each map is a homomorphism of graded $\coker w_1$-modules and $\de$ is a map of degree 0. Since they will appear so often in the text, we introduce the shorthand notation
\begin{equation}\label{CKH}
	C:=\coker w_1,\qquad H:={\rm H}^*(\OGr_k(n);\F_2),\qquad K:= \ker w_1,
\end{equation}
where $k$ and $n$ are usually fixed beforehand and clear from context. When $k$ and $n$ are not fixed, we use the notation $C_k(n)$ and $K_k(n)$.

For $k$ fixed, we will denote $W_1=\F_2[w_1\stb w_k]$ and $W_2=\F_2[w_2\stb w_k]$ (in most of the paper $k=3$). Then ${\rm H}^*(\Gr_k(n);\F_2)$ has a presentation as $W_1/(Q_{n-k+1}\stb Q_n)$ where $Q_i$ are some quotient Stiefel--Whitney classes (which can be expressed as Giambelli determinants) recalled in Section~\ref{sec:characteristic-subring}, and
\[C={\rm H}^*(\Gr_k(n);\F_2)/(w_1)=W_2/(q_{n-k+1}\stb q_n),\]
where $q_j=Q_j|_{w_1=0}$. Before proceeding further, let us fix some practical terminology.
\begin{definition}
	The subring $\pi^*C\se {\rm H}^*(\OGr_k(n);\F_2)$ is called the \emph{characteristic subring}, since it is the subring generated by the Stiefel--Whitney classes of the tautological bundle $S\to \OGr_k(n)$. In the terminology of \cite{KorbasRusin2016:rank}, a class $x\in H$ is \emph{anomalous}, if $x\not\in \im \pi^*$, equivalently, if its image under the projection $\delta\colon H\to K$ is nonzero. In abuse of terminology, we will therefore also call $K$ the \emph{anomalous module}.
\end{definition}

\begin{remark}
	We will be careful to distinguish between anomalous classes in $H$ and their image in $K$ in cases where it matters (such as for questions of the ring structure, since $\delta$ is only a $C$-module homomorphism). Namely, for $k=3$ and $2^{t-1}<n<2^t-3$, we will see in Proposition \ref{prop:K-relations} that $K_3(n)$ is generated by two elements $A_n$ and $D_n$; note that these denote elements of ${\rm H}^*(\Gr_3(n);\F_2)$. We will use $a_n$ and $d_n$ to denote lifts of these elements to ${\rm H}^*(\OGr_3(n);\F_2)$, i.e.\ these are elements satisfying $\de(a_i)=A_i$ and $\de(d_i)=D_i$.
\end{remark}

In connection with the question of computing the ring structure for $H$, we can ask what the $C$-module structure of $H$ is. Since \eqref{eq:w1ses} is a short exact sequence of graded $C$-modules, the $C$-module structure determines all products where at least one factor is in the image of the characteristic subring $C$. Knowing the $C$-module structure of $H$, all that remains to determine the ring structure of $H$ is the computation of products of anomalous classes in $H$. In fact, it suffices to compute products of classes which map to $C$-module generators of $K$ -- in the cases $k\leq 3$, there are at most two such generators.

To determine the $C$-module structure on the cohomology $H$ using the exact sequence \eqref{eq:w1ses}, one is faced with the following questions:
\begin{itemize}
	\item[a)]describe $C$ using generators and relations,
	\item[b)]describe $K$ as a $C$-module using generators and relations, and
	\item[c)]determine the extension class in $\Ext_C(K,C)$ given by the exact sequence \eqref{eq:introgysin}.
\end{itemize}

Question a) has been much addressed in the literature, often in terms of Gr\"obner bases. For the case $k=3$, we will answer question b) in Section~\ref{sec:K-presentation}. Question c) also turns out to have a simple answer for $k=3$ as the Ext-group actually vanishes in this case, cf.~Section~\ref{sec:ext-class}. However, we will see in Section~\ref{sec:beyondk3} that the Ext-group is in general non-trivial for $k\geq 4$, making question~c) significantly more difficult to answer in general.

\subsection{Comparison to spectral sequence methods}

While for the present, we will focus on using the Gysin sequence as a way to determine the cohomology of $\OGr_k(n)$, we briefly discuss other approaches that have been considered in the literature.

First, we note that complete information on the additive structure, i.e., information on the mod~2 Betti numbers, is available, cf.~\cite{Ozawa2021}. All the ways we know how to compute the mod~2 Betti numbers eventually boil down to understanding the multiplication with $w_1$ on ${\rm H}^*(\Gr_k(n),\mathbb{F}_2)$ in terms of Young diagram combinatorics. 

One way to understand the multiplicative structure is by use of suitable spectral sequences. To compute the mod $p$ cohomology of a general homogeneous space $G/H$, one can use the Eilenberg--Moore spectral sequence
\[
E_2^{*,*}={\rm Tor}_{{\rm H}^*({\rm B}H,\mathbb{F}_p)}^{*,*}({\rm H}^*({\rm B}G,\mathbb{F}_p);\mathbb{F}_p)\Rightarrow {\rm H}^*(G/H,\mathbb{F}_p)
\]
associated to the fiber sequence $G/H\to {\rm B}H\to{\rm B}G$ combined with the computation of cohomology of classifying spaces of Lie groups. This is the approach taken e.g.\ by Baum~\cite{Baum1968} and Franz~\cite{Franz2021}. 

However, there are still problems with the cohomology ring structure. While the spectral sequence has a multiplicative structure, this only implies that there is a filtration on ${\rm H}^*(G/H,\mathbb{F}_p)$ whose associated graded is computed by the $E_\infty$-page. As shown in \cite{Baum1968} and \cite{Franz2021}, these extension problems can be solved (and the extensions are split) in many cases, but for the particular case of $p=2$ there seem to be no general methods to solve the extension problems at this point. In the context of the spectral sequences, what we do in this paper is compute the multiplicative structure of the $E_\infty$-page (that's the $C$-module structure of $K$) and solve the extension problem (that's the vanishing of the Ext-group) for the specific case of $\OGr_3(n)$.

Alternatively, one could also use the Serre spectral sequence associated to the degree 2 covering $\OGr_k(n)\to \Gr_k(n)$ to compute the cohomology of $\OGr_k(n)$. However, in this case, the Serre spectral sequence in fact degenerates to the Gysin sequence. Similarly, in the case of the Eilenberg--Moore spectral sequence, the filtration on cohomology only has one nontrivial step. This means that in the case $\OGr_k(n)$, the spectral sequences available don't have more information than the Gysin sequence, and solving the extension problem requires different methods anyway. For this reason, we will work with the Gysin sequence throughout.

\subsection{Poincar\'e duality}
In some simpler cases, we can also exploit the Poincar\'e duality structure to establish relations between $C$ and $K$.

\begin{proposition}\label{prop:pd}
  Poincar\'e duality on $\Gr_k(n)$ induces a perfect pairing between $K=\ker w_1$ and $C=\coker w_1$. In particular, if $c_d$, $k_d$ denote the Betti numbers of $C$ and $K$ in degree $d$, respectively, then $c_i=k_{N-i}$, where $N$ is the dimension of $\Gr_k(n)$.
\end{proposition}

\begin{proof}
  We denote by
  \[
  b(x,y)\colon {\rm H}^*(\Gr_k(n),\mathbb{F}_2)\otimes {\rm H}^*(\Gr_k(n),\mathbb{F}_2)\to \F_2\colon (x,y)\mapsto \pi_!(x\cup y)
  \]
  the perfect pairing of Poincar\'e duality. Restrict $b(x,y)$ to $\ker w_1\otimes {\rm H}^*(\Gr_k(n),\mathbb{F}_2)$. The orthogonal complement of $\ker w_1$ under $b$ is $(w_1)$: by definition, $(w_1)\subset(\ker w_1)^\perp$, but then we have $\dim_{\mathbb{F}_2}\ker(w_1)+\dim_{\mathbb{F}_2}(w_1)=\dim_{\mathbb{F}_2}{\rm H}^*(\Gr_k(n),\mathbb{F}_2)$, forcing equality. Thus $b$ descends to a perfect pairing $\ker w_1\otimes \coker w_1\to \F_2$.
\end{proof}

\begin{remark}
  This is a purely algebraic and a more general statement; we just need a Poincar\'e duality algebra and an element in it.
\end{remark}

\begin{lemma}
  \label{lem:powerful}
  Let $A$ be a Poincar\'e duality $\F_2$-algebra, with a generator $\om_A$ of the top degree. Let $I\subseteq A$ be an ideal and denote by $C:=A/I$ the quotient algebra, and by $K:=\op{Ann}_AI$ the annihilator of $I$ . Assume that $C$ is also a Poincar\'e duality algebra with top degree generator $\om_C$. If $\om_C\cdot x=\om_A$ for some $x\in K$, then the morphism $\mu_x\colon C\to K$ given by multiplication with $x$ is injective.
\end{lemma}

\begin{remark}
  This is a surprisingly powerful lemma: in the cases where $K$ is a free $C$-module (generated by one element), we can apply it. In the cases when $K$ is not free, $C$ is not a Poincar\'e duality algebra - however, surprisingly, the condition $\om_C\cdot x=\om_A$ is still sometimes satisfied.
\end{remark}

\begin{corollary}
  \label{cor:powerful}
In the situation $\OGr_k(n)$, if the anomalous module $K$ is generated by one element, then $K$ is free of rank one. Consequently, ${\rm H}^*(\OGr_k(n),\mathbb{F}_2)\cong C\oplus K\cong C^{\oplus 2}$ as $C$-modules (with the second isomorphism ignoring the grading). 
\end{corollary}

\begin{proof}
  We can consider the total dimension of $K$ as $\mathbb{F}_2$-vector space. Since by assumption $K$ is cyclic, this is at most the total dimension of $C$ as $\mathbb{F}_2$-vector space. The perfect pairing from Proposition~\ref{prop:pd} then implies that $K$ needs to be free: any relation divided out would reduce the total dimension. 
  From this, we deduce that the extension \eqref{eq:w1ses} splits and the second claim follows.
\end{proof}

\begin{remark}
  This covers many situations considered in the literature, such as the cases $n=2^t-3,\dots,2^t$ for $\OGr_3(n)$ in \cite{BasuChakraborty2020,ColovicPrvulovic2023_2t,JovanovicPrvulovic2023}, or the cases ${\rm Gr}_4(n)$ for $n=8,9$ in \cite{Rusin2019}. We also find several further situations in Section~\ref{sec:beyondk3} such as $\OGr_5(16)$, $\OGr_5(32)$ and $\OGr_4(n)$ for $n=13,\dots,17,29,\dots,33$. One could conjecture that this generalizes to $\OGr_5(2^t)$ and $\OGr_4(n)$ with $2^t-3\leq n\leq 2^t+1$.
\end{remark}

\section{The characteristic subring and its properties} 
\label{sec:characteristic-subring}

In this short section, we will recall some information on the characteristic subring $C=\coker(w_1)$ as a ring and its presentation in terms of Giambelli determinants $q_i$. 

\subsection{Presentation of the characteristic subring}
Recall that the cohomology of the (unoriented) real Grassmannian $\Gr_k(n)$ has a presentation as
\[
{\rm H}^*(\Gr_k(n);\F_2)=\F_2[w_1\stb w_k]/(Q_{n-k+1}\stb Q_n),
\]
where the $Q_j$ are uniquely determined by the Whitney sum formula:
\[
1+w_1+\ldots +w_k=\frac{1}{1+Q_1+Q_2+\ldots }.
\]
More explicitly, the $Q_j$ can be written as the following Giambelli determinant:
\begin{equation}\label{eq:giambelli}
Q_j=\det
\left(\begin{array}{ccccc}
	w_1 & w_2 & \ldots & w_{j-1} & w_j \\
	1 & w_1 & w_2 & \ldots & w_{j-1} \\
	\vdots & \ddots & \ddots & \ddots & \ldots \\
	0 & \ldots & 1 & w_1 & w_2 \\
	0 & \ldots & 0 & 1 & w_1
\end{array}\right)	
\end{equation}

For the cokernel $C=\coker(w_1)$, being the quotient by $w_1$, we get a similar presentation \[C=\F_2[w_2\stb w_k]/(q_{n-k+1}\stb q_n),\] where now the $q_j$ are determined by
\begin{equation}\label{eq:qj}
	1+w_2+\ldots +w_k=\frac{1}{1+q_1+q_2+\ldots }.
\end{equation}
Alternatively, we obtain $q_j$ explicitly by setting $w_1=0$ in the above Giambelli determinant: 
\begin{equation}\label{eq:qjfromQj}
	q_j=Q_j|_{w_1=0}.
\end{equation}
This implies that $q_j$ satisfy the recursive formula
\begin{equation}\label{eq:recursion}
	q_j=\sum_{l=2}^k w_lq_{j-l}
\end{equation}
Via a standard computation, \cite[(2.8)]{ColovicPrvulovic2023_2t}, we can also write even more explicitly
\begin{equation}\label{eq:multinomial}
  q_j=\sum_{j=2a_2+\ldots + ka_k} \binom{|a|}{a}w^a,
\end{equation}
where $a=(a_2\stb a_k)$, $w^a=\prod_{i=2}^k w_i^{a_i}$, $|a|=\sum_{i=2}^k a_i$ and $\binom{|a|}{a}$ is the multinomial coefficient corresponding to $a$ mod~2.

\subsection{Ring-theoretic properties of $C$}

In this short section, we establish some ring-theoretic properties of $C$ that will be needed later in the description of Koszul homology.

\begin{proposition}
  \label{prop:dim0}
The commutative $\mathbb{F}_2$-algebra $C$ is a local ring of Krull dimension $0$. 
\end{proposition}

\begin{proof}
  It is clear that $(w_2,\dots,w_k)$ is an ideal consisting only of nilpotent elements. The complement consists of polynomials in Stiefel--Whitney classes which have constant coefficient 1. But because of the nilpotence of $w_2,\dots,w_k$, such polynomials are invertible. This means that $C$ is a local ring with maximal ideal $(w_2,\dots,w_k)$.

  Since the maximal ideal consists only of nilpotent elements, we also immediately get the claim about Krull dimension: the ideal $(w_2,\dots,w_k)$ is the nilradical and thus the unique prime ideal. For another way to see the claim about Krull dimension, note that $C$ is a quotient of the cohomology ring ${\rm H}^*(\Gr_k(n),\mathbb{F}_2)$ of the corresponding ordinary Grassmannian. The latter is the quotient of $\mathbb{F}_2[w_1,\dots,w_k]$ by the regular sequence $(q_{n-k+1},\dots,q_n)$. In particular, it is a complete intersection ring of Krull dimension $0$, since it is cut out by $k$ equations in $k$-dimensional affine space. Since $C$ is a quotient, it must also have Krull dimension $0$. 
\end{proof}

\begin{corollary}
  \label{cor:radical}
  For any $k,n$ we have the following equality of ideals in $W_2=\mathbb{F}_2[w_2,\dots,w_k]$:
  \[
  \sqrt{(q_{n-k+1},\dots,q_n)}=(w_2,\dots,w_k)
  \]
\end{corollary}

\begin{proof}
  By Proposition~\ref{prop:dim0}, $C=W_2/(q_{n-k+1},\dots,q_n)$ is a local ring of Krull dimension 0. It therefore has a single prime ideal, which is the maximal one $(w_2,\dots,w_k)$. Therefore, the unique prime ideal of $W_2$ containing $(q_{n-k+1},\dots,q_n)$ is the irrelevant ideal $(w_2,\dots,w_k)$, which therefore must be the radical. 
\end{proof}

\section{Syzygies of the characteristic subring}
\label{sec:syzygies}
In this section, we give a system of syzygies between $(q_{n-k+1}\stb q_n)$ over $W_2=\F_2[w_2\stb w_k]$.  First, we will describe one such relation for arbitrary $k$ and $n=2^t$ which plays a fundamental role. More precisely, we prove in Theorem~\ref{thm:twopower_syzygies}:
\[
\sum_{i\geq 0} w_{2i}q_{n-2i}=0.
\] 
From this relation we recover the relations discovered by Fukaya \cite{Fukaya} and Korba\v s \cite{Korbas2015} in their study of the cohomological properties of $\OGr_3(n)$ and $\OGr_4(n)$.  

In the second half of this section, we will examine how these relations give further relations as $n$ increases or decreases - we call these \emph{ascending} and \emph{descending relations}, cf.~\eqref{eq:descending_relation} and \eqref{eq:ascending_relation}. In this way we obtain relations between $(q_{n-k+1}\stb q_n)$ for all $n$. We will use this method in our computation of the Koszul homology of the $\OGr_3(n)$ case. The discussion in Section~\ref{sec:beyondk3} suggests that the fundamental syzygy and the procedure of ascending and descending relations might play a key role also in $k>3$ cases.

\subsection{Multinomial coefficients mod~2}
We first review some properties of multinomial coefficients mod~2 which will be relevant for the proof of the fundamental syzygy in Theorem~\ref{thm:twopower_syzygies}.

For a sequence of non-negative integers $a=(a_2\stb a_k)$\footnote{To agree with later notation we start the indexing from $a_2$.} we say that \emph{their base-2 expansions are disjoint} if each power of 2 appears in the base-2 expansion of at most one of the $a_j$. We will denote the multinomial coefficient corresponding to the sequence $a$ by $\binom{|a|}{a}$, where $|a|=\sum_{i=2}^k a_i$. The following characterization of mod~2 multinomial coefficients is also known as Lucas' theorem.
\begin{lemma}
  \label{lem:multinomial}
  For a tuple $a=(a_2,\dots,a_k)$, denote $|a|=\sum_{i=2}^k a_i$. Then the 2-adic valuation of the multinomial coefficient $\binom{|a|}{a_2,\dots,a_k}$ equals the number of carrying operations in the addition $|a|=a_2+\cdots+a_k$. As a consequence, the following are equivalent:
  \begin{enumerate}
  \item 
    \[
    \binom{|a|}{a_2,\dots,a_k}\equiv 1\bmod 2
    \]
  \item There is no carrying in the addition $|a|=a_2+\cdots+a_k$ in base 2.
  \item The base 2 expansions of the $a_i$ are disjoint.
  \end{enumerate}
\end{lemma}

\begin{proof}
  The Legendre formula for the 2-adic valuation of a factorial $n!$ states that $\nu_2(n!)=n-s_2(n)$ where $s_2$ is the sum of the digits in the base-2 expansion of $n$. The latter is simply the number of 1s in the base-2 expansion of $n$. So, for the multinomial coefficient, we compute the 2-adic valuation as follows:
  \begin{align*}
    \nu_2\binom{|a|}{a_2,\dots,a_k}&=\nu_2\left(\frac{|a|!}{a_2!\cdots a_k!}\right) =\nu_2(|a|!)-\nu_2(a_2!\cdots a_k!)\\
    &=|a|-s_2(|a|)-\left(|a|-s_2(a_2)-\cdots-s_2(a_k)\right)=\sum_{i=2}^ks_2(a_i)-s_2(|a|).
  \end{align*}
  We see that the 2-adic valuation of the multinomial coefficient is the sum of nonzero digits in the $a_2,\dots,a_k$ minus the nonzero digits in $|a|$. But this is equal to the number of carry operations (since each carry operation reduces the number of nonzero digits by 1). The equivalent characterizations of mod~2 nonvanishing of the multinomial coefficient follows directly from this.
\end{proof}

\begin{example}
  The base-2 expansion of $a=(2,5,8)$ is $[a]_2=(10,101,1000)$. Their sum is $1111$, which needs no carrying. The corresponding multinomial coefficient is $\binom{|a|}{a}=135135$ which is odd.	
\end{example}

For a given sequence $a=(a_2\stb a_k)$, let $\hat a_j$ denote the sequence \[\hat{a}_j=(a_2\stb a_j-1\stb a_k).\] 
We will need a lemma about `consecutive` multinomial coefficients mod~2:
\begin{lemma}\label{lem:consecutive_multinomial}
  The following relations hold for consecutive multinomial coefficients mod~2:
  \begin{enumerate}
  \item If $\binom{|a|}{a}\equiv 1\bmod 2$, then there is a unique $l$, such that $\binom{|a|-1}{\hat a_l}\equiv 1$. If $2^p$ is the largest two-power dividing all $a_i$, then $l$ is the unique index for which $2^{p+1}$ does not divide $a_l$.
  \item If $\binom{|a|}{a}\equiv 0\bmod 2$ and $\binom{|a|-1}{\hat a_j}\equiv 1\bmod 2$, then there is a unique $l\neq j$, such that $\binom{|a|-1}{\hat a_l}\equiv 1\bmod 2$. If $2^p$ is the largest two-power dividing all $a_i$, then $l$ is the unique index besides $j$ for which $2^{p+1}$ does not divide $a_l$.
  \end{enumerate}
\end{lemma}

\begin{proof}
  (1) By Lemma \ref{lem:multinomial}, the numbers $a_2,\dots,a_k$ have disjoint base-2 expansions. Decreasing one such number by 1 means changing the right-most digit 1 in the base-2 expansion to a 0, and changing all the zeros to the right of this position to 1s. The only way this doesn't destroy disjointness between the base-2 expansions is if we do this to the right-most digit 1 in any of the base-2 expansions of $a_2,\dots,a_k$. Since by assumption the base-2 expansions of the $a_2,\dots,a_k$ are disjoint, there is a unique $j$ where the right-most digit 1 appears. This argument also implies the statement about the largest two-power dividing $a_i$.
	
  (2) By Lemma \ref{lem:multinomial}, the base-2 expansions of $(a_2\stb a_j-1\stb a_k)$ are disjoint, but the base-2 expansions of $(a_2\stb a_j\stb a_k)$ are not. In $a_j$, there is exactly one nonzero bit which was zero in $a_j-1$ (this holds for any number, the nonzero bit is the last carry).  Since the $(a_2\stb a_j-1\stb a_k)$ are disjoint, but $(a_2\stb a_j\stb a_k)$ are not, the new nonzero bit in $a_j$ is nonzero in exactly one other $a_l$. So the largest two-power $2^p$ dividing $a_j$ also divides $a_l$, but $2^{p+1}$ does not. Therefore  $(a_2\stb a_l-1\stb a_k)$ are also disjoint.
\end{proof}

\begin{example}
  In the base-2 expansions of $a=(2,5,8)$, if we decrease 10, we get overlap between 1 and 101 and the corresponding multinomial coefficient is even. If we decrease 1000, we get 111 and overlap with both 10 and 101, and an even multinomial coefficient. Only if we decrease 101, we get 100 which doesn't overlap with any of the others. 
	
	From the lemma, for any nonzero multinomial coefficient we get a unique chain of decreasing multinomial coefficients which are nonzero mod~2:
	\[(2,5,8)\to (2,4,8)\to(1,4,8)\to (4,8)\to(3,8)\to\ldots\to (0,8)\]
\end{example}

The following proposition is the key step in proving the relation between the $q_i$'s.
\begin{proposition}\label{prop:multinomial_sums}
  If for a sequence $a=(a_2\stb a_k)$, we have $\sum_{i=2}^kia_i=2^t$, then
  \[\binom{|a|}{a} = \sum_{j=1}^{\lfloor k/2 \rfloor}\binom{|a|-1}{\hat{a}_{2j}} \]
\end{proposition}

\begin{proof}
  There are two cases depending on the parity of $\binom{|a|}{a}$.
	
  If $\binom{|a|}{a}\equiv 1\bmod 2$, then by part (1) of Lemma~\ref{lem:multinomial}, there is a unique $l$ such that $\binom{|a|-1}{\hat{a}_{l}}$ is also $\equiv 1\bmod 2$. We claim that $l$ is even, so that a unique term on the right-hand side is nonzero. If $2^p$ is the largest 2-power dividing all $a_i$, then $a_l$ is the unique $a_i$ not divided by $2^{p+1}$ by Lemma~\ref{lem:multinomial}. If $l$ is odd, then $\sum_{i=2}^k ia_i\equiv 2^p\bmod 2^{p+1}$. Since the left-hand side is equal to $2^n$, this implies $2^p=2^n$, but this is impossible since $l>1$.
	
  Assume now that $\binom{|a|}{a}\equiv 0\bmod 2$. It is enough to show that if $\binom{|a|-1}{\hat{a}_{2j}}\equiv 1$, then there is a unique $l$, such that $\binom{|a|-1}{\hat{a}_{l}}\equiv 1$ and that $l$ is even. The unique existence of $l$ is given by part (2) of Lemma~\ref{lem:multinomial}. We now show that $l$ is even. By part (2) of Lemma~\ref{lem:multinomial}, $2^p$ divides both $a_{2j}$ and $a_l$, but $2^{p+1}$ does not. If $l$ is odd, then since $2j$ is even, $\sum_{i=2}^kia_i\equiv 2^p$ mod $2^{p+1}$, and we can conclude as above.
\end{proof}

\subsection{Explicit relations between $q_i$}
For the description of the Koszul homology, we need to understand the $W_2=\F_2[w_2\stb w_k]$-relations between the $q_j$. Using the above results on mod~2 multinomial coefficients, we now obtain the following syzygy between the $q_i$'s for $n=2^t$ and arbitrary $k$:

\begin{theorem}\label{thm:twopower_syzygies}
  The following relation between the $q_j$ holds if and only if $n=2^t$:
  \begin{equation}\label{eq:two_power_relation}
    \sum_{i \text{ even}}q_{n-i}w_{i}=\sum_{i>1 \text{ odd}} q_{n-i}w_{i}=0.
  \end{equation}
\end{theorem}

\begin{proof}
  Assume that $n=2^t$. By the recursion \eqref{eq:recursion}, we have $\sum_{i=0}^k w_iq_{n-i}=0$, so it is enough to prove the statement about the even part. Since $q_j=\sum_{j=2a_2+\ldots +ka_k} \binom{|a|}{a} w^a$, the relation \eqref{eq:two_power_relation} reduces to a statement about multinomial coefficients, which is Proposition \ref{prop:multinomial_sums}.
	
  For the other direction, first assume that $n$ is odd. We can check that the sum in question is $w_3$ in case $n=3$, so we may assume $n\geq 5$. In that case, we can write $n=i+j$ for $i\geq 2$ even and $j\geq 3$ odd. We claim that the monomial $w_iw_j$ appears in the sum in question. Since $a_i=a_j=1$ and all other $a_l$ are 0, the multinomial coefficient is $0\bmod 2$, hence $w_iw_j$ doesn't appear in $q_n$. Now $w_j$ appears in $q_j$ in the sum in question, but $w_i$ doesn't appear in any $q_{n-2i}$ because $j$ is odd. So the sum in question contains the monomial $w_iw_j$ and is therefore not 0.
	
  For any $m$ which is not a power of 2, we can write $m=2^tn$ for $n$ odd. Now we know from the previous argument that the monomial $w_iw_j$ (with $n=i+j$) appears in the sum for $n$. But then the monomial $w_i^{2^t}w_j^{2^t}$ appears in the sum for $m=2^ti+2^tj$ and we are done.
\end{proof}

We can consider specializations of this formula to different $k$'s via the substitutions $w_{>k}=0$.  The following proposition can be found in \cite[Proposition~3.2]{Fukaya} and \cite[Lemma 2.3 (ii)]{Korbas2015}. We will show how the statements follow from Theorem~\ref{thm:twopower_syzygies}:
\begin{proposition}\label{prop:qvanishing}
  \begin{enumerate}
  \item For $k=2$, $q_{2t+1}=0$ for all $t$.
  \item For $k=3$, $q_{2^t-3}=0$ for all $t$. 
  \item For $k=4$, $q_{2^t-3}=0$ for all $t$.
  \end{enumerate}
\end{proposition}

\begin{proof}
  Statement (1) is trivial. For $k=3$, Theorem \ref{thm:twopower_syzygies} implies
  \[q_{2^n}+w_2q_{2^n-2}=w_3q_{2^n-3}=0.\]
  Since $\F_2[w_2,w_3]$ is a UFD, $q_{2^n-3}=0$. For $k=4$, Theorem \ref{thm:twopower_syzygies} implies
  \[q_{2^n}+w_2q_{2^n-2}+w_4q_{2^n-4}=w_3q_{2^n-3}=0,\]
  which implies $q_{2^n-3}=0$ as before. 
\end{proof}

We will show in the next sections that for $k=3$, all the syzygies between $(q_{n-2},q_{n-1},q_n)$ are obtained from the vanishing statement in Proposition~\ref{prop:qvanishing} via ascending and descending relations. However, in the cases $n=2^t-3,\dots,2^t$, one of these relations is ``inessential'', i.e., is contained in the image of the boundary map of the Koszul complex and therefore doesn't contribute to the presentation of $K$ as $C$-module. As we will discuss in Section~\ref{sec:beyondk3}, the $k=4$ case is somewhat similar in that most of the time we have two relations produced from Proposition~\ref{prop:qvanishing} via ascending and descending relations. 

\subsection{Descending relations}
In what follows, $k$ is fixed, and we will assume that $q_i$ is a sequence in $W_2$ (or a $W_2$-module) satisfying the recursive relationship:
\begin{equation}\label{eq:q_recursion}
	q_{i}=\sum_{j=2}^k w_j q_{i-j}.
\end{equation}

Then given a relation between $(q_{n-k+1}\stb q_n)$, there is also a relation between $(q_{n-k}\stb q_{n-1})$ as the following proposition describes. For our application, these $q_i$'s are the relations defining the characteristic subrings $C_k(n)$ and $C_k(n-1)$, respectively.

\begin{proposition}
  If $(q_i)$ satisfies \eqref{eq:q_recursion} and
  \[\sum_{j=0}^{k-1}\al_j^0 q_{n-j}=0\]
  is a homogeneous relation in degree $d$ between $(q_{n-k+1}\stb q_n)$, then
  \begin{equation}\label{eq:descending_relation}
    \sum_{j=0}^{k-1}\al_j^i q_{n-i-j}=0
  \end{equation}
  is a homogeneous relation in degree $d$ between $(q_{n-i-k+1}\stb q_{n-i})$. Here, the sequence of doubly-indexed polynomials $(\al_0^i)$ is defined by a double recursion starting with $\al_0^{<0}=0$, a recursion for the $\alpha_0^i$ given by 
  \begin{equation}\label{eq:al0_recursion}
    \al_{0}^{i+1}=\sum_{r=1}^{k-1}w_{k+1-r}\al_0^{i+r-k},
  \end{equation}
  and a recursion for the $\al_r^i$ defined by
  \begin{equation}\label{eq:alphaij_alpha_0j}
    \al_{k-j}^{i}=\sum_{r=0}^{j-1}w_{k-r}\al_0^{i+r-j}.
  \end{equation}
\end{proposition}

\begin{proof}
  Using the recursion \eqref{eq:q_recursion} for $q_n$, write
  \[
  \sum_{j=0}^{k-1}\al_j^0 q_{n-j}=
  \sum_{j=1}^{k}(\al_0^0w_j+\al_j^0)q_{n-j}=\sum_{j=0}^{k-1}(\al_0^0w_{j+1}+\al_{j+1}^0)q_{n-1-j}
  \]
  with $w_1=0$ and $\al^0_k=0$. By induction, setting \begin{equation}\label{eq:alphaij_recursion}
    \al^{i+1}_j=w_{j+1}\al^{i}_0+\al^i_{j+1}
  \end{equation}
  for $j=0\stb k-1$ and $\al_k^{i+1}=0$ we get a relation between $(q_{n-i}\stb q_{n-i-k+1})$:
  \[\sum_{j=0}^{k-1}\al_j^i q_{n-i-j}=0.\]
  Using \eqref{eq:alphaij_recursion},	the coefficients $\al_j^i$ can also be expressed recursively from the sequence $(\al_0^i)$:
  \[
  \al_{k-1}^{i+1}=w_k\al_0^i,\qquad \al_{k-2}^{i+1}=w_{k-1}\al_0^i+\al_{k-1}^{i}=w_{k-1}\al_0^i+w_{k}\al_{0}^{i-1},
  \]
  and by induction, we obtain \eqref{eq:alphaij_alpha_0j}.
\end{proof}

\begin{corollary}
  \label{cor:descending-k3}
  For $k=3$, the following relation in degree $m=2^t-3$ holds for all $n\leq m$ and $i=m-n$:
  \begin{equation}\label{eq:descending_3}
    q_iq_{n}+q_{i+1}q_{n-1}+w_3q_{i-1}q_{n-2}=0
  \end{equation}
  This is a relation between the relations $(q_n,q_{n-1}, q_{n-2})$ defining the characteristic subring $C_3(n)$.
\end{corollary}

\begin{proof}
  For $k=3$, by Fukaya's theorem (Proposition~\ref{prop:qvanishing}), $q_{2^t-3}=0$ is a relation. We set $\al_0^0=1=q_0$, $\al_1^0=\al_2^0=0$, $\al_0^{<0}=0$. The recursion \eqref{eq:al0_recursion} is $\al_{0}^{i}=w_3\al_0^{i-3}+w_2\al_0^{i-2}$. This is the same recursion defining $(q_i)$, cf.~\eqref{eq:q_recursion}, so that $\al_0^i=q_i$ for all $i$. Then by \eqref{eq:alphaij_alpha_0j},
  \[
  \al_1^i=w_3\al_0^{i-2}+w_2\al_0^{i-1}=w_3q_{i-2}+w_2q_{i-1}=q_{i+1},
  \]
  and
  \[\al_2^i=w_3\al_0^{i-1}=w_3q_{i-1}
  \]
  Therefore the relation \eqref{eq:descending_relation} is of the form \eqref{eq:descending_3}.
\end{proof}

\begin{corollary}
  \label{cor:descending_4}
  For $k=4$, the following relation in degree $m=2^t-3$ holds for all $n\leq m$ and $i=m-n$:
  \begin{equation}\label{eq:descending_4}
    q_iq_{n}+q_{i+1}q_{n-1}+(q_{i+2}+w_2q_{i})q_{n-2}+(q_{i+3}+w_2q_{i+1}+w_3q_i)q_{n-3}=0
  \end{equation}
  This is a relation between the relations $(q_n,q_{n-1}, q_{n-2}, q_{n-3})$ defining the characteristic subring $C_4(n)$.
\end{corollary}
\begin{proof}
  For $k=4$, by Proposition \ref{prop:qvanishing}, $q_{2^t-3}=0$ is a relation, so set $\al_0^0=1=q_0$, $\al_1^0=\al_2^0=\al_3^0=0$, $\al_0^{<0}=0$. The recursion \eqref{eq:al0_recursion} is $\al_{0}^{i}=w_4\al_0^{i-4}+w_3\al_0^{i-3}+w_2\al_0^{i-2}$,	which is again the same recursion as the one defining $(q_i)$, cf.~\eqref{eq:q_recursion}, so that $\al_0^i=q_i$ for all $i$. By \eqref{eq:alphaij_alpha_0j}
  \[
  \al_1^i=w_4\al_0^{i-3}+w_3\al_0^{i-2}+w_2\al_0^{i-1}=w_4q_{i-3}+w_3q_{i-2}+w_2q_{i-1}=q_{i+1},
  \]
  \[
  \al_2^i=w_4\al_0^{i-2}+w_3\al_0^{i-1}=w_4q_{i-2}+w_3q_{i-1}=q_{i+2}+w_2q_{i},
  \]
  and
  \[
  \al_3^i=w_4\al_0^{i-1}=q_{i+3}+w_2q_{i+1}+w_3q_{i}.
  \]
  Therefore the relation \eqref{eq:descending_relation} is of the form \eqref{eq:descending_4}.
\end{proof}

\subsection{Ascending relations}

Given a relation between $(q_{n-k+1}\stb q_n)$, we can also get a relation between $(q_{n-k+2}\stb q_{n+1})$ from the recursion $q_i=\sum_{j=2}^kw_jq_{i-j}$, as the following proposition describes. Again, for our application, the $q_i$'s are the relations defining the characteristic subrings $C_k(n)$ and $C_k(n+1)$.

\begin{proposition}
  \label{prop:ascending}
  If $(q_i)$ satisfies \eqref{eq:q_recursion} and
  \[\sum_{j=0}^{k-1}\be_j^0 q_{n-j}=0\]
  is a homogeneous relation in degree $D$ between $(q_{n-k+1}\stb q_n)$, then
  \begin{equation}\label{eq:ascending_relation}
    \sum_{j=0}^{k-1}\be_j^i q_{n+i-j}=0
  \end{equation}
  is a homogeneous relation in degree $D+ki$ between $(q_{n+i-k+1}\stb q_{n+i})$, where the sequence of doubly-indexed polynomials $(\be_j^i)$ is defined by a double recursion starting with  $\beta^{<0}_j=0$ for all $j$:
  \begin{equation}\label{eq:betaij_beta0j}
    \beta^{i+1}_j=\sum_{r=0}^jw_{j-r}w_k^r\be_{k-1}^{i-r}
  \end{equation}
  Note that $w_1=0$ and that \eqref{eq:betaij_beta0j} for $j=k-1$ is a recursive definition of $(\be_{k-1}^i)$:
  \[
  \beta^{i+1}_{k-1}=\sum_{r=0}^{k-1}w_{k-1-r}w_k^r\be_{k-1}^{i-r}
  \]
\end{proposition}

\begin{proof}
  We multiply the given relation with $w_k$, and use the recursion \eqref{eq:q_recursion} in the form $w_kq_{n-k+1}=\sum_{j=0}^{k-1}w_jq_{n+1-j}$ (with $w_1=0$ and $w_0=1$) to get
  \[
  w_k\sum_{j=0}^{k-1}\be_j^0 q_{n-j}=\be_{k-1}^0\left(\sum_{j=0}^{k-1}w_jq_{n+1-j}\right)+w_k\sum_{j=1}^{k-1}\be_{j-1}^0 q_{n+1-j}
  =\sum_{j=0}^{k-1}\left(\be_{k-1}^0w_j+w_k\be_{j-1}^0\right)q_{n+1-j}
  \]
  with $w_1=0$ and $\be^0_{-1}=0$. By induction, setting \begin{equation}\label{eq:betaij_recursion}
    \be^{i+1}_j=w_{j}\be^{i}_{k-1}+w_k\be^i_{j-1}
  \end{equation}
  for $0=1\stb k-1$ and $\be_{-1}^{i+1}=0$ we get a relation between $(q_{n+i-k+1}\stb q_{n+i})$:
  \[\sum_{j=0}^{k-1}\be_j^i q_{n+i-j}=0.\]
  Using \eqref{eq:betaij_recursion},	the coefficients $\be_j^i$ can also be expressed recursively from the sequence $(\be_0^i)$, starting with 
  \[
  \be_0^{i+1}=\be_{k-1}^i,\qquad \be_1^{i+1}=w_1\be_{k-1}^i+w_k\be_{k-1}^{i-1}.
  \]
  By induction on $j$, we obtain \eqref{eq:betaij_beta0j}.
\end{proof}

Applying this to Fukaya's relations $q_{2^t-3}=0$ in the case $k=3$, we directly get the following syzygy between the relations $(q_n,q_{n-1}, q_{n-2})$ defining the characteristic subring $C_3(n)$. To alleviate notation, for $k=3$ we introduce 
\begin{equation}\label{eq:r}
	r_j=\be_2^j.
\end{equation}

\begin{corollary}
  \label{cor:ascending-k3}
  For $k=3$, let $d=2^{t}-1$, $n\geq d$ and $j=n-d$. Then we have the following relation between $(q_{n},q_{n-1}, q_{n-2})$ in degree $(d-2)+3j$:
  \begin{equation}\label{eq:ascending_3}
    r_{j-1}q_{n}+w_3r_{j-2}q_{n-1}+r_jq_{n-2}=0
  \end{equation}
  where $r_i$ satisfy the recursion
\begin{equation}\label{eq:r_recursion}
  r_{i+1}=w_2r_i+w_3^2r_{i-2}.	
\end{equation}
  with $r_0=1$, $r_{<0}=0$. 
\end{corollary}
\begin{proof}
  We first check the degrees in the statement. Assuming the recursion, we see that the degree of $r_j$ is $2j$, and then the relation has degree $2j-2+n=2n-2d-2=(d-2)+3j$.
  
  Now we prove the statement by induction. The induction start is $n=d$, $j=0$. In this case, $r_{-1}q_n+w_3r_{-2}q_{n-1}+r_0q_{n-2}=0$ holds because $r_{<0}=0$ and $q_{2^t-3}=0$ from Proposition~\ref{prop:qvanishing}. The recursion $r_{i+1}=w_2r_i+w_3^2r_{i-2}$ is exactly the last recursion from Proposition~\ref{prop:ascending}. To see that the claim follows by induction, there is some reindexing necessary so that we can write the starting relation as $\beta_0^0q_n+w_3\beta_1^0q_{n-1}+\beta_2^0q_{n-2}=0$ and use Proposition~\ref{prop:ascending}. In particular, we start with $\beta_0^0=\beta_1^0=0$ and $\beta_2^0=1$. From \eqref{eq:betaij_beta0j}, we get $\beta_0^j=w_0w_3^0r_{j-1}$ and $\beta_1^j=w_0w_3r_{j-2}$ (using $w_1=0$). In particular, we find
  \[
  \beta_0^jq_n+\beta_1^jq_{n-1}+\beta_2^jq_{n-2}=r_{j-1}q_n+w_3r_{j-2}q_{n-1}+r_jq_{n-2}.
  \]
  This finishes the proof.
\end{proof}

We remark that from the recursion \eqref{eq:ascending_3}, via a standard computation, $r_j$ can also be given in a closed form as follows:
\begin{equation}\label{eq:rj_closed}
	r_j=\sum_{2j=2b_2+6b_3}\binom{b_2+b_3}{b_2}w_2^{b_2}w_3^{2b_3}
\end{equation}

Similarly, for $k=4$, we get the following syzygy between the relations $(q_n,\dots,q_{n-3})$ defining the characteristic subring $C_4(n)$.

\begin{corollary}
  \label{cor:ascending-k4}
  For $k=4$, let $d=2^r$, $n\geq d$ and $i=n-d$. Then we have the following relation between $(q_{n},q_{n-1}, q_{n-2},q_{n-3})$ in degree $(d-3)+4i$:
  \begin{equation}\label{eq:ascending_4}
    \be_{3}^{i-1}q_{n}+w_4\be_3^{i-2}q_{n-1}+\left(w_2\be_3^{i-1}+w_4^2\beta_3^{i-3}\right)q_{n-2}+\beta_3^iq_{n-3}=0
  \end{equation}  
  where $\beta_3^i$ satisfy the recursion
  \[
  \beta_3^{i+1}=w_3\beta_3^i+w_2w_4\beta_3^{i-1}+w_4^3\beta_3^{i-3}
  \]
  with $\beta_3^0=1$ and $\beta_3^{<0}=0$.
\end{corollary}

\begin{proof}
  We first consider the degrees of the relation in the statement. From the recursion for $\beta_3^i$, we find that the degree of $\beta_3^i$ is $3i$. The relation then has degree $3i-3+n=2n-3d-3=d-3+4i$.
  
  We start with the base case $n=d$, $i=0$. In this case, the first three terms have $\beta_3^{<0}=0$ while the last term is $q_{2n-3}$, so the relation holds.

  For general $n=d+i$, Proposition~\ref{prop:ascending} implies a relation of the form
  \[
  \beta_0^iq_n+\beta_1^iq_{n-1}+\beta_2^iq_{n-2}+\beta_3^iq_{n-3}=0. 
  \]
  Now we can rewrite the coefficients using the recursions from Proposition~\ref{prop:ascending}, omitting the terms containing $w_1=0$:
  \begin{align*}
    \beta_0^i&=w_0w_4^0\beta_3^{i-1}\\
    \beta_1^i&=w_0w_4^1\beta_3^{i-2}\\
    \beta_2^i&=w_2w_4^0\beta_3^{i-1}+w_0w_4^2\beta_3^{i-3}
  \end{align*}
  Writing out the last recursion in Proposition~\ref{prop:ascending} produces $\beta_3^{i+1}=w_3w_4^0\beta_3^i+w_2w_4^1\beta_3^{i-1}+w_0w_4^3\beta_3^{i-3}$. The claims follow from this.
\end{proof}

\section{Generalities on Koszul homology and the anomalous module}
\label{sec:koszul-homology}
In this section, we prepare for the computation of the $C$-module presentation of $K$ by interpreting $K$ as first Koszul homology for the sequence $(q_{n-k+1},\dots,q_n)$ in $W_2$. That Koszul homology should play a role for the computation is already apparent in Baum's computation of cohomology of homogeneous spaces in \cite{Baum1968}. 

\subsection{Recollections on Koszul homology}

\begin{definition}\label{def:koszul}
  For a commutative ring $R$, and an $R$-linear map $f\colon R^n\to R$, the associated \emph{Koszul complex} $\mathcal{K}(f)$ is the complex
  \[
  \xymatrix{
    \bigwedge^nR^n\ar[r]^-{\wedge^n f}&\bigwedge^{n-1}R^{n}\ar[r]^-{\wedge^{n-1} f}&\cdots\ar[r]^-{\wedge^2 f}& \bigwedge^1R^n\ar[r]^-{\wedge^1 f}&\bigwedge^0R^n\iso R
  }
  \]
  with boundary maps
  \[
  \wedge^jf(\al_1\wedge\cdots \wedge \al_j)=\sum_{i=1}^j(-1)^{i+1}f(\al_i)\cdot \al_1\wedge\cdots\wedge\hat{\al}_i\wedge\cdots \wedge \al_j.
  \]
  If $Q_1\stb Q_n\in R$ is a sequence of elements in $R$, then the associated Koszul complex $\mathcal{K}(Q)$ is the Koszul complex of the $R$-linear map $(r_1\stb r_n)\mapsto \sum_{i=1}^nQ_ir_i$.

  For an $R$-module $M$, we then define the \emph{Koszul homology} as the homology of the complex $\mathcal{K}_M(f):=\mathcal{K}(f)\otimes_R M$. Below, we will only consider the situation where $f$ is given by a sequence $Q=(Q_1,\dots,Q_n)$, and the Koszul homology will be denoted by $H_i(Q,M)$. 
\end{definition}

\subsection{Relation of Koszul homology to $K$}\label{sec:Koszul_and_K}

Let $W_1:=\F_2[w_1\stb w_k]$ and $W_2:=\F_2[w_2\stb w_k]$. Then $W_2$ is naturally a $W_1$-module (where $w_1$ acts by 0) and there is a short exact sequence of $W_1$-modules:
\begin{equation}\label{eq:SES}
	\xymatrix{
		0\ar[r]& W_1\ar[r]^{w_1}& W_1\ar[r]& W_2\ar[r]& 0.
	}
\end{equation}
Let $H:={\rm H}^*(\Gr_k(n);\F_2)$, which has the presentation $H=W_1/(Q)$ with $Q=(Q_{n-k+1},\dots,Q_n)$. Since $Q$ is a regular sequence in $W_1$, all Koszul homologies $H_i(Q, W_1)$ vanish for $i>0$. In general, for a $W_1$-module $M$, \[H_0(Q, M)=M/(Q)M\] (this is one of the defining properties of Koszul homology). Therefore
\[H_0(Q,W_2)=W_2/(Q){W_2}=C,\qquad H_0(Q,W_1)=W_1/(Q){W_1}=H\]
The short exact sequence of $W_1$-modules \eqref{eq:SES} induces a long exact sequence of $W_1$-modules in Koszul homology:
\begin{equation}\label{eq:KoszulLES}
  \xymatrix@R-1pc{
    & 0\ar[r]& 0\ar[r]& H_1(Q,W_2)\ar `r/7pt[d] `[l] `^dl[ll] `^r[r] [dll]&\\
    & H\ar[r]^{w_1}& H\ar[r]& C\ar[r]&0
  }
\end{equation}
Therefore \[K=\ker (w_1\colon H\to H)=\de(H_1(Q,W_2)),\]
and the exactness of the sequence also proves that $H_i(Q,W_2)=0$ for $i>1$.

\begin{remark}
	In fact, $H_i(Q,W_2)=\op{Tor}_i^{W_1}(H,W_2)$, see e.g.\ \cite[Corollary 4.5.5.]{weibel} and \eqref{eq:KoszulLES} is the long exact sequence of $\op{Tor}_i^{W_1}(H,\cdot)$ associated to \eqref{eq:SES}.
\end{remark}

The way to unravel the boundary map $\de$ is via the snake lemma: take a relation $r_i\in H_1(Q,W_2)$ (between $q_i$), lift it to some element $R_i$ in the Koszul complex $(\mathcal{K}_{W_1}(Q))_1$, take its boundary $d_1(R_i)\in (\mathcal{K}_{W_1}(Q))_0=W_1$, take a preimage in $W_1$ via $w_1$ and take the its reduction to $H=W_1/(Q)$. Since the sequence is exact, that element is in the kernel of $w_1$. 

\emph{This procedure gives explicit $W_1$-generators of $K$ in terms of the syzygies (relations between relations) in $C$.} Also, the presentation of $K$ as a $W_1$-module provides a presentation as a $C$-module, by base extension. Indeed, $K\otimes_{W_1}C=K$, since $\Ann_{W_1}(K)=(w_1,Q_{n-k+1}\stb Q_n)$ -- in other words, this is just the statement that $K$ is a $C$-module.
\begin{remark}\label{rmk:explicit}
	We can make this boundary map even more explicit. Let $Q_i\in W_1$ be the relations (the complete homogeneous symmetric polynomials expressed in terms of elementary symmetric ones) and let 
	\[q_i:=Q_i|_{w_1=0}\in W_2,\qquad  P_i:=Q_i+q_i\in(w_1).\] 
	Since $P_i\in (w_1)\se W_1$, it has a unique preimage $p_i$ via $w_1$, i.e.\ $w_1p_i=P_i$. That is,
	\[
	\xymatrix@R-2pc{
		W_1\ar[r]^{w_1}& W_1\ar[r]& W_2\\
		p_i\ar[r]& P_i& \\
		&Q_i\ar[r]& q_i& 
	}
	\]
	If $\sum f_iq_i=0\in W_2$ is a relation, then since $Q_i=0$ in $H$, we have the following equality in $H$:
	\[\sum f_iP_i=\sum f_i q_i=0,\] since $\sum f_iq_i=0\in W_2\se W_1$, and using that $H$ is a quotient of $W_1$. 
	The boundary map in the long exact sequence~\eqref{eq:KoszulLES} maps such a relation $\sum f_iq_i$ in the Koszul complex to
	\begin{equation}\label{eq:LESboundary}
		\de\left(\sum f_iq_i\right) =\sum f_ip_i\in H.
	\end{equation}
	Note that since $w_1p_i=P_i$, $\sum f_i p_i$ is an element in $\ker w_1\se H$. Also note that $\de$ decreases the degrees by one; $\deg \de(\xi)=\deg \xi-1$ for $\xi\in \ker(d_1)$.
	
	The classes $P_i$ and $p_i$ also satisfy recursive identities: since $Q_i=\sum_{r=1}^k w_rQ_{i-r}$ and $q_i=\sum_{r=2}^k w_rq_{i-r}$,
	\[P_i=Q_i+q_i=w_1Q_{i-1}+\sum_{r=2}^k w_rP_{i-r}\]
	and since $Q_{i-1}=P_{i-1}+q_{i-1}$:
	\begin{equation}\label{eq:Precursion}
		P_i=w_1q_{i-1}+\sum_{r=1}^k w_rP_{i-r}	
	\end{equation}
	
		So 
	\begin{equation}\label{eq:precursion}
		p_i=q_{i-1}+\sum_{r=1}^kw_rp_{i-r}
	\end{equation}
\end{remark}
\begin{remark}
	The Koszul homology of $(q_{n-2},q_{n-1},q_n)\se W_2=\mathbb{F}_2[w_2,w_3]$ is the same as the Koszul homology of $(w_1,Q_{n-2},Q_{n-1},Q_n)\subseteq W_1=\mathbb{F}_2[w_1,w_2,w_3]$.
\end{remark}
\begin{remark}
	The Koszul homology description generalizes to the cohomology of the sphere bundle $S$ of a rank $n$ vector bundle $E\to X$, whenever ${\rm H}^*(X)$ is a complete intersection ring $\F_2[x_1\stb x_r]/(Q_1\stb Q_p)$ for some regular sequence $Q_i$. Then the first homology of the Koszul complex of $(w_n(E),Q_1\stb Q_p)$ over $\F_2[x_1\stb x_r]$ is  $\ker w_n(E)\se {\rm H}^*(X)$.
\end{remark}

\begin{remark}
  \label{rem:graded}
  In what follows, we will in fact consider $W_1$ resp. $W_2$ as graded rings, with the grading given by $\deg w_i=i$. The characteristic subring $C$ and the anomalous module both have gradings coming from cohomological degree compatible with this. The Koszul complex is then also a complex of graded modules, and the identification of $K$ as first Koszul homology is compatible with this grading. Some of the computations in Section~\ref{sec:K-presentation} will make use of this grading.
\end{remark}

\subsection{The characteristic rank of $\OGr_k(2^t)$}
The \emph{characteristic rank} of a vector bundle $E\to X$ is the maximal degree $r$, such that ${\rm H}^{\leq r}(X;\F_2)$ is generated by the Stiefel--Whitney classes of $E$. Then for the tautological bundle $S\to \OGr_k(n)$,
$\crk(S)$ is the maximal degree $r$ such that the inclusion $C^{\leq r}\subset {\rm H}^{\leq r}(\OGr_k(n);\mathbb{F}_2)$ is an isomorphism. So a nonzero class in the Koszul homology $H_1(Q,W_2)$ in degree $d$ gives the upper bound $\crk(S)\leq d-2$.  The characteristic ranks $\crk(S)$ for $k=3$ and $k=4$ have been determined in \cite[Theorem 1]{PetrovicPrvulovicRadovanovic} and \cite[Theorem 6.6]{PrvulovicRadovanovic2019}. The characteristic rank $\crk(S)$ is not known besides these cases, though upper and lower bounds have been developed in the literature,see the discussion below. Using Theorem \ref{thm:twopower_syzygies} we can give a new upper bound for $\crk(S)$ when $n=2^t$.

\begin{theorem}\label{thm:charrank}
  Assume $2^t-5\geq k\geq 5$. The characteristic rank of the tautological bundle $S\to \OGr_k(2^t)$ satisfies
  \[\crk(S)\leq 2^t-2.\] 
  In particular, there is a nonzero anomalous class in the cohomology  ${\rm H}^{2^t-1}(\OGr_k(2^t);\F_2)$ of the oriented Grassmannian.
\end{theorem}

\begin{proof}
We can assume $k\leq 2^{t-1}$ by the duality $\OGr_k(2^t)\iso \OGr_{2^t-k}(2^t)$. The relation \[\sum_{i=1}^{\lfloor{k/2}\rfloor-1} q_{n-2i-1}w_{2k+1}=0\] from Theorem \ref{thm:twopower_syzygies} gives a generator of $\ker(d_1)$ in the Koszul complex: 
\[
v=(0,0,0,w_3,0,w_5,0,\ldots)\in \wedge^1W_2^{\oplus k}.
\]
We claim that this is not in the image of the differential $d_2$ which has the following form:
\[
\wedge^2(W_2^{\oplus k})\ni A=\left(\begin{array}{ccccc}
	0 & \la_{12} & \la_{13} & \ldots & \la_{1k} \\
	\la_{12} & 0 & \la_{23} & \ldots & \la_{2k} \\
	\vdots & \vdots & \ddots & \ldots & \vdots \\
	\la_{1k} & \la_{2k} & \la_{3k} & \ldots & 0
\end{array}\right)\mapsto A
\left(
\begin{array}{c}
	q_{n-k+1} \\
	q_{n-k+2} \\
	\vdots\\
	q_n
\end{array}
\right)\in\wedge^1(W_2^{\oplus k})
\]
Since all entries of $Aq$ are homogeneous and the lowest degree nonzero term in $q$ is of degree $n-k+1\geq 2^t-2^{t-1}+1\geq 3$, this implies that $v\in \ker(d_1)$ is not in $\im(d_2)$, and defines a nontrivial element in the Koszul homology $H_1(Q,W_2)$. Via the long exact sequence \eqref{eq:KoszulLES}, this defines a nonzero element of $K$, which in turn lifts to an anomalous class in ${\rm H}^{2^t-1}(\OGr_k(2^t);\F_2)$. Therefore the characteristic rank is at most $2^t-2$, by definition.
\end{proof}

This anomalous class in $\OGr_k(2^t)$ provides further anomalous classes for $\OGr_k(n)$ by ascending and descending the relation of Theorem \ref{thm:twopower_syzygies}. Sometimes these ascended and descended relations vanish in Koszul homology, as they fall into the image of the Koszul boundary, as we illustrate in Section \ref{sec:koszulboundary}. However, these classes always seem to be responsible for the characteristic rank of $S\to \OGr_k(n)$, cf.\ Section \ref{sec:beyondk3}. We formulate our observations in the following conjecture:
\begin{conjecture}
  \label{new-amazing-conjecture}
  The characteristic rank of the tautological bundle $S\to\OGr_k(n)$ is equal to
  \begin{equation}\label{eq:charrank}
  	\crk(S)=\min (2^t-2,k(n-2^{t-1})+2^{t-1}-2).
  \end{equation}
  for $5\leq k\leq 2^{t-1}<n\leq 2^t$ and $t\geq 5$.
\end{conjecture}
We verified this conjecture for $k=5$, $n\leq 32$ and $k=6$, $n\leq 23$, cf.~Section~\ref{sec:beyondk3}. 
\begin{remark}
	To explain the $t\geq 5$ assumption, note that for $k\leq 4$, the syzygies follow a different pattern which gives rise to different characteristic ranks (cf.\ Proposition \ref{prop:qvanishing} and Corollaries \ref{cor:descending-k3} to \ref{cor:ascending-k4}). This also influences the characteristic ranks via the duality $\OGr_k(n)\iso \OGr_{n-k}(n)$ in the range $8<n\leq 16$ for $k=5,6$. The explicit examples not satisfying \eqref{eq:charrank} are $\crk(S\to\OGr_5(10))=10$, $\crk(S\to\OGr_5(11))=13$ and $\crk(S\to\OGr_6(12))=13$, see Section \ref{sec:beyondk3}. 
\end{remark}
\begin{remark}
  We briefly summarize what else is known about the characteristic rank of $S\to \OGr_k(n)$. There is a simple lower bound $\crk(S)\geq n-k-1$; this amounts to the obvious statement that there cannot be a $W_2$-relation between $(q_{n-k+1}\stb q_n)$ in degree $n-k$. The best known lower bounds have been obtained in \cite{PrvulovicRadovanovic2019} using Gr\"obner basis techniques; for $k\geq 6$:
  \[\crk(S)\geq (n-k)+\lfloor k/3\rfloor-1,\] 
  as well as some stronger estimates for certain pairs $(k,n)$. The authors also note \cite[Remark 6.8]{PrvulovicRadovanovic2019} that ``there are reasons to believe that for $k \geq 5$, there is some $n$ such that $\crk(S) > 2(n-k) - 1$''\footnote{The notation is adapted to comply with the notation of the present paper.}
  From the tables in Section \ref{sec:beyondk3}, we see that 
  \[\crk(S\to \OGr_6(18))=2(n-k)+4,\qquad  \crk(S\to \OGr_6(19))=2(n-k)+6.\]
  More generally, Conjecture~\ref{new-amazing-conjecture} would imply that for any $r$, there is a choice of $k$ and $n$, such that $\crk(S\to\OGr_k(n))>2(n-k)+r$.
\end{remark}

\begin{remark}
	In the literature, often the characteristic rank of a manifold is considered, which is the characteristic rank of its tangent bundle $\crk(TM)$. Whenever $w(TM)$ can be expressed in terms of $w(E)$, there is an obvious upper bound $\crk(TM)\leq \crk(E)$. In general, these are not equal; for example $\crk(T\RP^{2r+1})=0$, but $\crk(S\to\RP^{2r+1})=2r+1$. On the other hand, $\crk(T\RP^{2r})=\crk(S\to \RP^{2r})=2r$. By a result of Korba\v s \cite[p.\ 72]{Korbas2010}, $\crk(T\OGr_{k}(2n+1))=\crk(S\to\OGr_{k}(2n+1))$ - this is due to the fact that $w(S)$ can also be expressed from $w(T\OGr_{k}(2n))$.
\end{remark}

\section{Presentation of the anomalous module for \texorpdfstring{$k=3$}{k=3}}
\label{sec:K-presentation}

For $k=3$, we exhibited certain $W_2$-relations between $(q_{n-2},q_{n-1}, q_n)$ in Corollaries~\ref{cor:descending-k3} and \ref{cor:ascending-k3}. These involved the polynomials $q_i$ themselves and the polynomials $r_l$ of degree $2l$ defined by the recursion $r_{l+1}=w_2r_l+w_3^2r_{l-2}$. The goal of this section is to prove the following Theorem:
\begin{theorem}\label{thm:anomalousmodule}
  Let $2^{t-1}<n\leq 2^t-4$, and set $i=2^t-3-n$ and $j=n-2^{t-1}+1$. Then the anomalous module $K\se {\rm H}^*(\Gr_3(n);\F_2)$ is isomorphic (as a graded $C$-module) to
  \begin{equation}
    K=C\bra A_n,D_n\ket/(q_iA_n+r_{j-1}D_n,
    q_{i+1}A_n+w_3r_{j-2}D_n, 
    w_3q_{i-1}A_n+r_jD_n),
  \end{equation}
  where $A_n$ and $D_n$ are explicit elements \eqref{eq:ascended_generator}, \eqref{eq:descended_generator} of $K$ of degrees $\deg A_n=3n-2^t-1$ and $\deg D_n=2^t-4$. 
\end{theorem}

We will derive this presentation of $K$ using its identification as the first Koszul homology of the ideal $(q_{n-2},q_{n-1},q_n)\se W_2=\F_2[w_2,w_3]$ described in Section \ref{sec:Koszul_and_K}. Throughout this section, fix the indices $i,j,t,n$ as stated in the theorem, cf.\ Figure~\ref{fig:notations}.

\begin{figure}
	\begin{tikzpicture}
		\draw (0,0) -- (6,0);
		\fill[black] (0,0) circle (2pt);
		\fill[black] (6,0) circle (2pt);
		\fill[black] (4,0) circle (2pt);
		\draw (0,0) node[anchor=north] {$p=2^{t-1}-1$};
		\draw (4,0) node[anchor=north,inner sep=8pt] {$n$};
		\draw (6,0) node[anchor=north] {$m=2^t-3$};
		\draw (2,1) node[anchor=north] {$j$};
		\draw (5,1) node[anchor=north] {$i$};
		\draw[decoration={brace,raise=5pt},decorate]
		(0,0) -- node[above=6pt] {} (4,0);
		\draw[decoration={brace,raise=5pt},decorate]
		(4,0) -- node[above=6pt] {} (6,0);
	\end{tikzpicture}
	\caption{Notations}\label{fig:notations}
\end{figure}
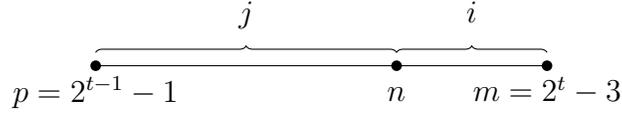

\subsection{Relations between $q_l$ and $r_l$}
The following equalities connect the coefficients of the ascended and descended relations, and will play a crucial role for the proofs in this section: 

\begin{lemma}
  \label{lem:linear-equations}
  With the indexing set up above, the following equations hold for all $i,j,n$:
    \begin{align*}
    w_3r_{j-2}q_i+r_{j-1}q_{i+1}&=q_{n-2}\\
    r_jq_i+w_3r_{j-1}q_{i-1} &=q_{n-1}\\
    r_jq_{i+1}+w_3^2r_{j-2}q_{i-1}&=q_n
  \end{align*}
\end{lemma}

\begin{proof}
The start of the induction is the case $n=2^{t-1}-1$, $j=0$ and $i=2^{t-1}-2$, where we can make use of $r_{<0}=0$ and $r_0=1$. In this case, the first equality reduces to $q_{2^{t-1}-3}=0$ from Proposition~\ref{prop:qvanishing}. The second and third equality both reduce to $q_i=q_{n-1}$ which follows from the $i=n-1$ of the assumption.
  
  For the induction step, assume that the three equations are satisfied for $n$, with $i=2^t-n-3$ and $j=n-2^{t-1}-1$. We want to show that the equations for $n+1$, $i-1$ and $j+1$ are also satisfied.

  We first show $r_{j+1}q_i+w_3^2r_{j-1}q_{i-2}=q_{n+1}$:
  \begin{align*}
    q_{n+1}&=w_2r_jq_i+w_2w_3r_{j-1}q_{i-1}+w^2_3r_{j-2}q_i+w_3r_{j-1}q_{i+1}\\
    &=r_{j+1}q_i+w_2w_3r_{j-1}q_{i-1}+w_3r_{j-1}q_{i+1}\\
    &=r_{j+1}q_i+w_2w_3r_{j-1}q_{i-1}+w_2w_3r_{j-1}q_{i-1}+w_3^2r_{j-1}q_{i-2}\\
    &=r_{j+1}q_i+w_3^2r_{j-1}q_{i-2}
  \end{align*}
  The first equality uses the recursion $q_{n+1}=w_2q_{n-1}+w_3q_{n-2}$, combined with the equations for the inductive assumption. The second equality combines the first and third summand using the recursion for $r_{j+1}$, and the third equality uses the recursion $q_{i+1}=w_2q_{i-1}+w_3q_{i-2}$. The two middle terms cancel and we have established the claim.

  Next, we show $r_{j+1}q_{i-1}+w_3r_{j}q_{i-2} =q_{n}$:
  \begin{align*}
    q_{n}&=r_jq_{i+1}+w_3^2r_{j-2}q_{i-1}\\
    &=w_2r_jq_{i-1}+w_3r_jq_{i-2}+w_3^2r_{j-2}q_{i-1}\\
    &=r_{j+1}q_{i-1}+w_3r_jq_{i-2}
  \end{align*}  
  The first equality is the inductive assumption, the second uses the recursion $q_{i+1}=w_2q_{i-1}+w_3q_{i-2}$, and the third uses the recursion for $r_{j+1}$, on the first and third summand in the second line. So the second equation is established.

  Finally, the relation $w_3r_{j-1}q_{i-1}+r_{j}q_{i}=q_{n-1}$ is one of the equations in the inductive assumptions.
\end{proof}

\subsection{The kernel of the Koszul differential $d_1$}
The next proposition shows that $\ker(d_1)$ is a free $W_2$-module of rank 2 and explicitly identifies a $W_2$-basis.
\begin{proposition}
  \label{prop:k3-basis}
  For $\OGr_3(n)$ with $2^{t-1}<n\leq 2^t-4$, with the above conventions $i=2^t-3-n$ and $j=n-2^{t-1}+1$, we have the following equality:
\[
\left(\begin{array}{ccc}q_{n-2}&q_{n-1}&q_n\end{array}\right)
  \left(\begin{array}{c}x\\y\\z\end{array}\right)
    =\det\left|\begin{array}{ccc}x & r_j&w_3q_{i-1}\\
    y&w_3r_{j-2}&q_{i+1}\\
    z&r_{j-1}&q_i\end{array}\right|    
\]
In particular, the 2nd and 3rd column vectors on the right-hand side form a basis for the kernel $\ker(d_1)$ in the Koszul complex, and so $\ker(d_1)$ is a free graded $W_2$-module of rank 2.
\end{proposition}    

\begin{proof}
  Before we embark on the proof, we recall from Remark~\ref{rem:graded} that the Koszul complex is a complex of graded modules. In particular, kernels of differentials as well as homology groups will inherit a grading, and all linear algebra arguments below will be with homogeneous elements
  
  (1) 
  The matrix equation claimed in the proposition is equivalent to the following three equations, which follow from Lemma~\ref{lem:linear-equations}:
  \begin{align*}
  q_{n-2}&=\det\left(\begin{array}{cc}w_3r_{j-2}&q_{i+1}\\r_{j-1}&q_{i}\end{array}\right)\\
  q_{n-1}&=\det\left(\begin{array}{cc}r_{j}&w_3q_{i-1}\\r_{j-1}&q_{i}\end{array}\right)\\
    q_{n}&=\det\left(\begin{array}{cc}r_{j}&w_3q_{i-1}\\w_3r_{j-2}&q_{i+1}\end{array}\right)
  \end{align*}

  For ease of reference in the next two steps, we denote the 2nd and 3rd column vectors on the right-hand side by $u=(r_j, w_3r_{j-2}, r_{j-1})^{\rm t}$ and $v=(w_3q_{i-1}, q_{i+1},q_i)^{\rm t}$.

  (2) As a next step, we want to show that $u$ and $v$ span a free graded $W_2$-submodule of $W_2^{\oplus 3}$ of rank 2, i.e., that for any $W_2$-linear combination $\lambda u+\mu v=0$ with homogeneous $\lambda,\mu$ we have $\lambda=\mu=0$. So assume we have such a $W_2$-linear combination. This means that the determinants of the three $(2\times 2)$-minors of the $(3\times 2)$-matrix $(u,v)$ are annihilated by $\lambda\mu$. By the matrix equation established in (1), these determinants are $q_{n-2}$, $q_{n-1}$ and $q_n$. Note that $\sqrt{(q_{n-2},q_{n-1},q_n)}=(w_2,w_3)$, by Corollary~\ref{cor:radical}. This means that for any homogeneous prime ideal $\mathfrak{p}\subseteq W_2=\mathbb{F}_2[w_2,w_3]$ which is not the irrelevant ideal $(w_2,w_3)$, at least one of $q_{n-2},q_{n-1}$ and $q_n$ is not contained in $\mathfrak{p}$, so that the annihilator of $(q_{n-2},q_{n-1},q_n)$ must be contained in $\mathfrak{p}$. In particular, $\lambda\mu$ is contained in the nilradical of $W_2$, which is $(0)$, i.e., one of $\lambda$ and $\mu$ is $0$.

  To conclude that the other coefficient also has to be 0, we need to know that $u$ and $v$ are both non-zero, then the claim follows since $W_2$ is an integral domain. From \cite[Lemma~2.3 (ii)]{Korbas2015}, we know that $q_j=0$ if and only if $j=2^t-3$. In particular, no three consecutive $q_j$ can be 0, and thus $v\neq 0$. Our assumption also implies that none of $q_{n-2}$, $q_{n-1}$ and $q_n$ are 0. The determinant formulas from step (1) then imply that no two $r$'s appearing in $u$ vanish at the same time, and thus $u\neq 0$. Therefore, if one of the coefficients $\lambda,\mu$ is 0, then so is the other one.
  
(3) Finally, we need to show that $u$ and $v$ actually span $\ker(d_1)$. It is easy to check that ${\rm span}\langle u,v\rangle\subseteq \ker(d_1)$. The matrix equation, established in step (1), implies that for a vector $s=(x,y,z)^{\rm t}\in\ker(d_1)$ in the kernel, the set $\{s,u,v\}$ is linearly dependent, i.e., there exist homogeneous $\alpha,\lambda,\mu\in W_2$ such that $\alpha s+\lambda u+\mu v=0$. If we can show that $\alpha\in W_2^\times$, we have $s\in{\rm span}\langle u,v\rangle$ and we are done. So assume that $\alpha$ is not invertible, then there exists a homogeneous maximal ideal $\mathfrak{m}\subseteq W_2$ such that $\alpha\in\mathfrak{m}$. After tensoring with the residue field $\mathbb{F}=W_2/\mathfrak{m}$, we get $\overline{\lambda}\overline{u}+\overline{\mu}\overline{v}=0$. The proof of (2) actually shows that this is impossible, since at least one of the determinants of the $(2\times 2)$-minors (i.e.\ $q_{n-2},q_{n-1},q_n$ with radical $(w_2,w_3)$) is nonzero in $\mathbb{F}$. 
\end{proof}

\begin{remark}\label{rmk:nonvanishing_r_coefficients}
  The proof actually shows that all the coefficients $r_{\geq 0}$ appearing in the relation are in fact nonzero. If one of them was nonzero, the basis $u,v$ would not span the kernel. But then some $u$ would be divisible by $w_3$, leading to a linear expression $w_3s+u=0$. The last step in the proof shows that this is not possible, and we in fact do get all the elements in the kernel.
\end{remark}

\subsection{Presentation of the anomalous module}
Using Proposition \ref{prop:k3-basis}, we can name explicit elements $A_n,D_n\in K_3(n)$ via the boundary map $\de$ of the long exact sequence of Koszul homologies \eqref{eq:KoszulLES}, as described in Remark \ref{rmk:explicit}, in particular \eqref{eq:LESboundary}. The first generator of $\ker(d_1)$ is the descended relation:
\[
q_iq_n+q_{i+1}q_{n-1}+w_3q_{i-1}q_{n-2}=0.
\]
Denote by $D_n\in W_1$ the image of this relation via $\de$:
\begin{equation}\label{eq:descended_generator}
	D_n=q_ip_n+q_{i+1}p_{n-1}+w_3q_{i-1}p_{n-2}.
\end{equation}
Similarly, denote by $A_n\in W_1$ the image of the ascended relation via $\de$:
\begin{equation}\label{eq:ascended_generator}
	A_n=r_{j-1}p_n+w_3r_{j-2}p_{n-1}+r_jp_{n-2}.
\end{equation}
In order not to overburden notation, we will also denote by $A_n,D_n$ the reductions of these classes from $W_1$ to ${\rm H}^*(\Gr_3(n);\F_2)=W_1/(Q_{n-2},Q_{n-1},Q_n)$.
From Proposition \ref{prop:k3-basis} and the long exact sequence \eqref{eq:KoszulLES} it follows that $K_3(n)$ is generated by these two elements as a $C$-module, however so far these elements could just as well be 0.
In this section, we will show that this is not the case, and compute the relations between these elements.

\begin{proposition}
  \label{prop:K-relations}
  The graded $W_2$-module $K$ has a presentation by 2 generators $A_n$ and $D_n$ in degrees $\deg A_n=3n-2^t-1$ and $\deg D_n=2^t-4$. These two generators satisfy 3 relations of degrees $2n-4, 2n-3$ and $2n-2$, respectively, given as follows:
  \[
  q_iA_n+r_{j-1}D_n, \qquad
  q_{i+1}A_n+w_3r_{j-2}D_n, \qquad
  w_3q_{i-1}A_n+r_jD_n.
  \]
  A similar presentation holds for $K$ as graded $C$-module.
\end{proposition}

\begin{proof}
  We first note that the degrees actually are as claimed: this follows from \eqref{eq:descended_generator}, \eqref{eq:ascended_generator}, $\deg(p_l)=\deg(q_l)-1=l-1$ and $\deg(r_l)=2l$. 
  
  The differential $d_2$ for the Koszul complex is given by the following matrix, cf.~Definition~\ref{def:koszul}:
  \[
  \left(\begin{array}{ccc}
    q_{n-1}&q_n&0\\
    q_{n-2}&0&q_n\\
    0&q_{n-2}&q_{n-1}
  \end{array}\right)
  \]
  From Proposition~\ref{prop:k3-basis}, we know that $\ker d_1$ is a free rank 2 module. A presentation of $K$, as the first homology of the Koszul complex, can thus be obtained by writing the columns of the above matrix in terms of the basis vectors given in Proposition~\ref{prop:k3-basis}. We thus need to solve the following system of $W_2$-linear equations:
  \[
  \left(\begin{array}{ccc}
    q_{n-1}&q_n&0\\ q_{n-2}&0&q_n\\ 0&q_{n-2}&q_{n-1}
  \end{array}\right)
  =\left(\begin{array}{cc}
    r_j&w_3q_{i-1}\\ w_3r_{j-2}&q_{i+1}\\ r_{j-1}&q_i
  \end{array}\right)\cdot
  \left(\begin{array}{ccc}
    \lambda_{11}&\lambda_{12}&\lambda_{13}\\
    \lambda_{21}&\lambda_{22}&\lambda_{23}    
  \end{array}\right)
  \]
  The columns in the right-most matrix will then describe the coefficients of the three relations between the two generators.

  The claim in the proposition is that the following matrix is a solution of the system:
  \[
  \left(\begin{array}{ccc}
    \lambda_{11}&\lambda_{12}&\lambda_{13}\\
    \lambda_{21}&\lambda_{22}&\lambda_{23}    
  \end{array}\right)=
  \left(\begin{array}{ccc}
    q_i&q_{i+1}&w_3q_{i-1}\\
    r_{j-1}&w_3r_{j-2}&r_j
  \end{array}\right)
  \]
  One easily checks that the product has the symmetric structure required for the $d_2$-differential of the Koszul complex. It remains to check the following three equations, which are established in Lemma~\ref{lem:linear-equations}:
  \begin{align*}
    w_3r_{j-2}q_i+r_{j-1}q_{i+1}&=q_{n-2}\\
    r_jq_i+w_3r_{j-1}q_{i-1} &=q_{n-1}\\
    r_jq_{i+1}+w_3^2r_{j-2}q_{i-1}&=q_n\qedhere
  \end{align*}
\end{proof}
  This concludes the proof of Theorem \ref{thm:anomalousmodule}.
  \subsection{Koszul boundary}\label{sec:koszulboundary}
In this section we show how at the boundary of the interval $2^{t-1}<n<2^t-3$ the ascending and descending relations fall in the image of the Koszul boundary $d_2$. In other words we explain why there is only one generator in these cases.
\begin{proposition}

The descending relation \eqref{eq:descending_3}
\[    q_iq_{n}+q_{i+1}q_{n-1}+w_3q_{i-1}q_{n-2}=0\]	
is in the image of the Koszul boundary $d_2$ for $n=2^{t-1}$ and the corresponding $i=2^{t-1}-3$.

The ascending relation \eqref{eq:ascending_3}
\[
r_{j-1}q_{n}+w_3r_{j-2}q_{n-1}+r_jq_{n-2}=0
\]
is in the image of the Koszul boundary $d_2$ for $n=2^t-3$ and the corresponding $j=2^{t-1}-2$.
\end{proposition}

\begin{proof}
The descending relation is
\[    q_{2^{t-1}-3}q_{2^{t-1}}+q_{2^{t-1}-2}q_{2^{t-1}-1}+w_3q_{2^{t-1}-4}q_{2^{t-1}-2}=0\]
which using $w_3q_{2^{t-1}-4}=q_{2^{t-1}-1}$ and $q_{2^{t-1}-3}=0$ is
\[q_{2^{t-1}-2}q_{2^{t-1}-1}+q_{2^{t-1}-1}q_{2^{t-1}-2}=0,\]
which is clearly an element in the image of $d_2$. The ascending relation is the following (using $q_{2^t-3}=0$):
\[
w_3r_{2^{t-1}-4}q_{2^t-4}+r_{2^{t-1}-2}q_{2^t-5}=0
\]
We can conclude, since the coefficients are $q_{2^t-5}$ and $q_{2^t-4}$, by the following lemma.
\end{proof}

\begin{lemma}
	For all $t\geq 2$, the following relations hold:
	\[r_{2^{t-1}-2}=q_{2^t-4},\qquad w_3r_{2^{t-1}-4}=q_{2^t-5}.\]
The recursion on $r_n$ also implies $r_{2^{t-1}-1}=q_{2^t-2}$.
\end{lemma}

\begin{proof}
Using the recursion \eqref{eq:recursion}, one can show that the $q_j$ satisfy \cite[(2.6)]{Korbas2015} 
\[
q_j=w_2^{2^s}q_{j-2\cdot 2^s}+w_3^{2^s}q_{j-3\cdot 2^s},
\]
for all $s$ such that $j\geq 1+3\cdot 2^s$. By a similar inductive argument on $s$ using the recursion \eqref{eq:r_recursion}, one can show 
\[
r_j=w_2^{2^s}r_{j-2^s}+w_3^{2^{s+1}}r_{j-3\cdot 2^{s}}
\]
and similarly
\begin{equation}\label{eq:r_3times}
r_{3\cdot 2^{t}-2}=w_2^{2^{t}}r_{2^{t+1}-2}.
\end{equation}
Using this, we can show the first equality $r_{2^{t-1}-2}=q_{2^t-4}$ by induction. The first step for $t=2,3$ states that $r_0=q_0=1$ and $r_2=q_4=w_2^2$. Using the recursions \eqref{eq:recursion}, \eqref{eq:r_recursion} and \eqref{eq:r_3times}, the induction step is 
\begin{align*}
q_{2^t-4}=w_2^{2^{t-2}}q_{2^{t-1}-4}+w_3^{2^{t-2}}q_{2^{t-2}-4}
=w_2^{2^{t-3}}\underbrace{w_2^{2^{t-3}}r_{2^{t-2}-2}}_{r_{3\cdot 2^{t-3}-2}}+w_3^{2^{t-2}}r_{2^{t-3}-2}=r_{2^{t-1}-2}.
\end{align*}
The other case can be obtained by an entirely analogous argument.
\end{proof}

\section{Vanishing of Ext-groups for \texorpdfstring{$k=3$}{k=3}}
\label{sec:ext-class}

Given the presentation of $K$ as a $C$-module established in Section~\ref{sec:K-presentation}, we will now show that the $\Ext$-group ${\rm Ext}^1_C(K,C)$ always vanishes in the $k=3$ case, basically for degree reasons. This implies, in particular, that the extension \eqref{eq:introgysin} always splits, and ${\rm H}^*(\OGr_k(n),\mathbb{F}_2)$ is, as a $C$-module, simply isomorphic to $K\oplus C$. Some basics on Ext-groups, how to compute them and how they relate to extensions can be found in Appendix~\ref{sec:ext-basics}. Note that we are interested in graded degree 0 extensions, i.e., extensions of graded $C$-modules where all the maps preserve degrees.

\begin{proposition}
  \label{prop:ext-vanishing-k3}
  For all $n\geq 3$, we have ${\rm Ext}^1_C(K,C)=0$. 
\end{proposition}

\begin{proof}
  In the cases $n=2^t-3,\dots,2^t$, $K$ is a cyclic $C$-module by \cite[Theorem A]{BasuChakraborty2020}, and therefore free by the Corollary~\ref{cor:powerful}. This implies the triviality of the Ext-group.

  We can thus focus on the cases $2^{t-1}<n\leq 2^t-4$. In these cases, $K$ is generated by two elements of degrees $3n-2^t-1$ and $2^t-4$, again by \cite[Theorem A]{BasuChakraborty2020} or Proposition~\ref{prop:k3-basis}. The smallest-degree anomalous generator is Poincar\'e dual to the top degree class of $C$. The dimension of $\OGr_3(n)$ being $3n-9$, the top degree of $C$ is therefore the maximum of $2^t-8$ and $3n-2^t-5$.

  On the other hand, $K$ has 3 relations in degrees $2l+2=2n-4,2l+3,2l+4$, by Proposition~\ref{prop:K-relations}. If we can show that the degrees of the relations are always bigger than the top degree of $C$, the Ext-group is trivial for degree reasons, cf. the discussion in Appendix~\ref{sec:ext-basics}. But the assumption $n>2^{t-1}$ implies $2n-4>2^t-4>2^t-8$, and the assumption $n\leq 2^t-4<2^t+1$ implies $2n-4>3n-2^t-5$. So we are done.
\end{proof}

\section{Removing remaining ambiguities}
\label{sec:ambiguities}

Using the splitting of Proposition \ref{prop:ext-vanishing-k3}, we know that ${\rm H}^*(\OGr_3(n),\F_2)$ in the range $2^{t-1}<n\leq 2^t-4$ is a $C$-module generated by lifts $a_n$ and $d_n$ of the elements $A_n$ and $D_n$, and we have computed the $C$-module relations in Theorem \ref{thm:anomalousmodule}. This determines the cohomology ${\rm H}^*(\OGr_k(n),\mathbb{F}_2)\cong C\oplus K$ as a $C$-module. To determine the complete ring structure, the only ambiguities left are the products $a_n^2, d_n^2$ and $a_nd_n$. We prove:

\begin{proposition}\label{prop:products} 
	Let $a_n, d_n\in {\rm H}^*(\OGr_3(n);\F_2)$ be lifts of the elements $A_n,D_n\in {\rm H}^*(\Gr_3(n);\F_2)$ defined in \eqref{eq:ascended_generator} and \eqref{eq:descended_generator}. Then
	\[
	a_n^2=d_n^2=a_nd_n=0.
	\]
\end{proposition}
\begin{proof}
The product $a_nd_n$ vanishes for degree reasons: 
\[\deg(a_nd_n)=(3n-2^t-1)+(2^t-4)=3n-5>3n-9=\dim\OGr_3(n).\]
We will show that $\de(a_n)^2=A_n^2=0$ and $\de(d_n)^2=D_n^2=0$ in $ {\rm H}^*(\Gr_3(n);\F_2)$ in Propositions~\ref{prop:ascendedsquare} and \ref{prop:descendedsquare}. Then we can conclude by the following Lemma \ref{lemma:reduction}. 
\end{proof}

\begin{lemma}\label{lemma:reduction}
	$a_n^2=d_n^2=0$ if and only if $A_n^2=D_n^2=0$.
\end{lemma}
\begin{proof}
  That boundary morphisms commute with Steenrod squares is the stability of Steenrod operations. 
  For the other direction, it is enough to show that $\deg(A_n^2)$ and $\deg(D_n^2)$ are above the top degree of $C$. As in Proposition~\ref{prop:ext-vanishing-k3}, the top degree of $C$ is $\max\{2^t-8,3n-2^t-5\}$. 
	  In the following lines, we will check that (under the standing assumptions) the degrees of the squares are always bigger than the top degree of $C$, so we are done. The inequality $\deg(D_n^2)=2(2^t-4)=2^{t+1}-8>2^t-8$ is straightforward. Next, $\deg(A_n^2)=2(3n-2^t-1)>3n-2^t-5$ follows from \[3n-2^{t+1}+2^t+3>3\cdot 2^{t-1}-2^{t+1}+2^t+3=2^{t-1}+3>0.\] Similarly, $\deg(A_n^2)=2(3n-2^t-1)>2^t-8$ follows from $3n-2^t-1>2^{t-1}-1>2^{t-1}-4$. Finally
	  \[
	 \deg(D_n^2)= 2(3n-2^t-1)>2(3\cdot 2^{t-1}-2^t-1)=2(2^{t-1}-1)=2^t-2>2^t-8.\qedhere
	  \]
\end{proof}

In the remainder of this section, we complete the proof of Proposition \ref{prop:products} by showing that the squares of the $A_n$ and $D_n$ vanish in ${\rm H}^*(\Gr_3(n);\F_2)$.

\subsection{The descended square}
Recall that $P_i=Q_i+q_i$ and $P_i=w_1p_i$, and that they satisfy the following recursive formulas:
\begin{equation}\label{eq:recallrecursions}
	P_i=w_1q_{i-1}+\sum_{r=1}^k w_rP_{i-r},\qquad
	p_i=q_{i-1}+\sum_{r=1}^kw_rp_{i-r}.
\end{equation}		
To prove $D_n^2=0$ in ${\rm H}^*(\Gr_3(n);\F_2)$, we will use the following recursive identity for $D_n$:
\begin{lemma}
  \label{lem:dn-recursion}
  The $D_n$ satisfy the recursion
  \[D_{n-1}=q_iQ_{n-1}+D_n,\]
  with the usual notation, where $i+n$ is some constant.
\end{lemma}

\begin{proof}
  Recall from~\eqref{eq:descended_generator} that
  \[
  D_n=q_ip_n+q_{i+1}p_{n-1}+w_3q_{i-1}p_{n-2}.
  \]
  Using the recursion on $q_l$ and $w_2p_{n-2}+w_3p_{n-3}=w_1p_{n-1}+q_{n-1}+p_{n}$ -- see \eqref{eq:precursion} -- we have
  \begin{align*}
    D_{n-1}&=q_{i+1}p_{n-1}+q_{i+2}p_{n-2}+w_3q_{i}p_{n-3}\\
    &=q_{i+1}p_{n-1}+(w_2q_i+q_{i-1}w_3)p_{n-2}+w_3q_{i}p_{n-3}\\
    &=\underbrace{q_ip_n+q_{i+1}p_{n-1}+q_{i-1}w_3p_{n-2}}_{D_n}+q_i(\underbrace{w_1p_{n-1}+q_{n-1}}_{Q_{n-1}}).\qedhere
  \end{align*}
\end{proof}
 
We now show that the square of the descended generator is always zero in ${\rm H}^*(\Gr_3(n);\F_2)$.
\begin{proposition}\label{prop:descendedsquare}
	For $2^{t-1}<n\leq 2^{t}-3$,
	\[
	D_n^2\in (Q_n,Q_{n-1},Q_{n-2})
	\]
\end{proposition}

\begin{proof}
  We prove this by downwards induction on $n$. For $n=2^t-3$, the reduction of $D_{n}^2\in W_1$ to \[{\rm H}^*(\Gr_3(2^t-3);\F_2)=W_1/(Q_{2^t-3},Q_{2^t-4}, Q_{2^t-5})\] vanishes by the results of \cite[Theorem 1.1. (c)]{JovanovicPrvulovic2023}. Indeed, ${\rm H}^*(\OGr_3(2^t-3);\F_2)=C[d_n]/d_n^2$   
  and since $\de(d_n)=D_n$ and $\de$ is a $\Sq$-module homomorphism, it follows that \[D_n^2=\Sq^{\deg d_n}\de(d_n)=\de(d_n^2)=0.\] Therefore 
  \[D_{2^t-3}^2\in (Q_{2^t-3},Q_{2^t-4}, Q_{2^t-5}).\]
The induction step is then stated in the following lemma.
\end{proof}

\begin{lemma}
	If $D_n^2\in(Q_n,Q_{n-1},Q_{n-2})$, then $D_{n-1}^2\in (Q_{n-1},Q_{n-2},Q_{n-3})$.
\end{lemma}

\begin{proof}
	Write for some coefficients $d_0,d_1,d_2\in W_1$,
	\[
	D_n^2=d_2Q_{n-2}+d_1Q_{n-1}+d_0Q_n.
	\]
	We would like to show that there exist $d'_0,d'_1,d'_2\in W_1$, such that
	\[
	D_{n-1}^2=d'_2Q_{n-3}+d'_1Q_{n-2}+d'_0Q_{n-1}.
	\]
	By Lemma~\ref{lem:dn-recursion}, the induction hypothesis and the recursion for $Q_n$ analogous to \eqref{eq:recursion}, we have
	\begin{align*}
		D_{n-1}^2&=D_n^2+q_i^2Q_{n-1}^2=d_2Q_{n-2}+(d_1+q_i^2Q_{n-1})Q_{n-1}+d_0Q_n\\
		&=d_0w_3Q_{n-3}+(d_0w_2+d_2)Q_{n-2}+(d_1+q_i^2Q_{n-1}+w_1d_0)Q_{n-1}
	\end{align*}
	which tell us the coefficients $d_i'$.
\end{proof}

\subsection{The ascended square}
Similarly, set $p=2^{t-1}-1$ and $j=n-p=n-2^{t-1}+1$. Recall that the $r_j$ are defined by the recursion
\[r_{j+1}=w_2r_j+w_3^2r_{j-2}.\]
The $A_n$'s also satisfy a recursion:
\begin{lemma}
  \label{lem:an-recursion}
  The $A_n$ satisfy the recursion
  \[A_{n+1}=w_3A_n+r_jQ_n.\]
\end{lemma}

\begin{proof}
	By definition \eqref{eq:ascended_generator},
	\[
	A_n=r_{j-1}p_n+w_3r_{j-2}p_{n-1}+r_jp_{n-2},
	\]
  Since $Q_{n}=w_1p_n+q_{n}$ (see Remark \ref{rmk:explicit}), the recursion \eqref{eq:precursion} can be written as $p_{n+1}=Q_n+w_2p_{n-1}+w_3p_{n-2}$, and we have
  \begin{align*}
    A_{n+1}&=r_{j}p_{n+1}+w_3r_{j-1}p_{n}+r_{j+1}p_{n-1}\\
    &=r_j(Q_n+w_2p_{n-1}+w_3p_{n-2})+w_3r_{j-1}p_n+(w_2r_j+w_3^2r_{j-2})p_{n-1}\\
    &=w_3\underbrace{(r_{j-1}p_n+w_3r_{j-2}p_{n-1}+r_jp_{n-2})}_{A_n}+r_jQ_n\qedhere
  \end{align*}
\end{proof}

\begin{proposition}\label{prop:ascendedsquare}
  For $2^{t-1}\leq n<2^{t}-3$,
  \[
  A_n^2\in (Q_n,Q_{n-1},Q_{n-2})
  \]
\end{proposition}

\begin{proof}
  We prove this by induction on $n$. For $n=2^{t-1}$, the reduction of $A_{n}^2\in W_1$ to \[{\rm H}^*(\Gr_3(2^{t-1});\F_2)=W_1/(Q_{2^{t-1}},Q_{2^{t-1}-1}, Q_{2^{t-1}-2})\] vanishes by the results of \cite[Theorem 1.1.]{ColovicPrvulovic2023_2t} and the same arguments as in Proposition \ref{prop:descendedsquare}. 
  The induction step is contained in the following lemma.
\end{proof}

\begin{lemma}
  If $A_n^2\in(Q_n,Q_{n-1},Q_{n-2})$, then $A_{n+1}^2\in (Q_{n+1},Q_{n},Q_{n-1})$.
\end{lemma}

\begin{proof}
  Assume that in $W_1$ we can write
  \[
  A_n^2=a_2Q_{n-2}+a_1Q_{n-1}+a_0Q_n.
  \]
  Then we want to show that increasing $n$ and $j$, there exist $a'_0,a'_1,a'_2$, such that
  \[
  A_{n+1}^2=a'_2Q_{n-1}+a'_1Q_{n}+a'_0Q_{n+1}.
  \]
  By Lemma~\ref{lem:an-recursion}, $Q_{n+1}=w_1Q_{n}+w_2Q_{n-1}+w_3Q_{n-2}$ and the induction hypothesis:
  \begin{align*}
    A_{n+1}^2&=w_3^2A_n^2+(r_jQ_n)^2=w_3^2(a_2Q_{n-2}+a_1Q_{n-1}+a_0Q_n)+(r_jQ_n)^2\\
    &=w_3a_2(Q_{n+1}+w_1Q_n+w_2Q_{n-1})+w_3^2(a_1Q_{n-1}+a_0Q_n)+(r_jQ_n)^2\\
    &=(w_2w_3a_2+w_3^2a_1)Q_{n-1}+\big(r_j^2Q_n+w_3^2a_0+w_1w_3a_2\big)Q_n+(w_3a_2)Q_{n+1}
  \end{align*}
  which tell us the coefficients $a_i'$.
\end{proof}
This concludes the proof of Proposition \ref{prop:products} and thus the proof of Theorem \ref{thm:main}.

\section{Discussion of the \texorpdfstring{$k>3$}{k>3} cases}
\label{sec:beyondk3}

In this final section, we want to outline some of the issues that arise when we go beyond the $k=3$ case. Some of these issues are well-known, but some haven't been noticed because the module structure over the characteristic subring hasn't been investigated much.

\subsection{Ascending and descending relations}

The technique of ascending and descending relations works rather generally, as we discussed in Section~\ref{sec:syzygies}.  However, starting from $k=4$, there are many further syzygies besides the ones obtained from Theorem \ref{thm:twopower_syzygies}. For now, it is not clear what to expect. As a first glance, we can use the \verb!Macaulay2! code in the appendix to check the degrees of generators of the anomalous module $K$ for $k=5$, starting with $n=10$:

\begin{center}
\begin{tabular}{|l|l||l|l||l|l||l|l|}
  \hline
  $n$ & degrees & $n$ & degrees & $n$ & degrees & $n$ & degrees\\
  \hline
  10 & 11, 13 &   16 & 15 &  22 & 31, 39, 40, 41, 42, 45&   28 & 31, 48, 51\\
  11 & 14, 15, 16 &   17 & 20, 24&   23 & 31, 40, 42, 42, 50&  29 & 31, 48\\
  12 & 15, 16, 19 &   18 & 25, 27, 29&   24 & 31, 40, 42, 47, 55&  30 & 31, 53\\
  13 & 15, 16 &  19 & 30, 30, 31, 32&   25 & 31, 40, 42&  31 & 31, 54 \\
  14 & 15, 21 &  20 & 31, 33, 35, 35, 40&   26 & 31, 43, 45& 32 & 31 \\
  15 & 15, 22&   21 & 31, 36, 38, 40, 40&   27 & 31, 46, 48& 33 & 36, 56\\
  \hline
\end{tabular}
\end{center}

We make some observations:
\begin{itemize}
\item 
  In the two cases $n=16$ and $n=32$, the module $K$ is free of rank 1 on the generator described in Theorem \ref{thm:charrank} (which subsequently implies that ${\rm H}^*(\OGr_5(2^t),\mathbb{F}_2)$ is a free $C$-module of rank 2 for $t=4$ and $t=5$).
\item These generators provide generators for other values of $n$ as predicted by the ascending and descending relations. Between $n=16\stb 24$, the ascending generators live in the following degrees: $(15,20,25\stb  55)$. The situation is similar as the one described for $k=3$: the ascended relation for $n=25$ in degree 60 is in the image of the Koszul boundary, and therefore represents 0, which is why it is not visible in this table. Similarly, we have descending generators for $n=32,31\stb 18$ in degree 31. 
  However, these generators describe only a small portion of all the generators of the anomalous module. 
\end{itemize}

We can make a similar table for $k=6$, recording the degrees of the generators of $K$ for $n=12$ to $n=21$: 
\begin{center}
\begin{tabular}{|l|l||l|l|}
  \hline
  $n$ & degrees & $n$ & degrees\\
  \hline
  12 & 14, 15, 16  &   17 & 21, 26, 28 \\
  13 & 15, 16 &   18 & 27, 28, 30, 32, 36 \\
  14 & 15, 16 &   19 & 31, 33, 34, 34, 36, 38, 38, 40, 42 \\
  15 & 15, 22 &   20 & 31, 38, 38, 38, 39, 40, 40, 40, 42, 42\\
  16 & 15, 22 &   21 & 31, 40, 42, 42, 44, 44, 44, 45, 46, 46\\
  \hline
  \end{tabular}
\end{center}
The previous phenomenon of a single generator for the case $n=2^t$ seems to disappear beyond $k=5$. The number of generators seems to grow. Some stabilization patterns (both for varying $n$ with fixed $k$ and with varying $k$) are discernible, but the rules of the game seem unclear for now. Nevertheless, the prevalence of one generator in degree $2^t-1$ seems to persist and the picture supports Conjecture~\ref{new-amazing-conjecture}. 

The ascending relations would move $k$ steps each time. One family of such generators for $k=5$ is visible, starting with the degree 20 generator for $n=17$, another one for $k=6$ starting with degree 21 in $n=17$. But most of the degrees do not seem to follow easy patterns compatible with ascending and descending relations.

\subsection{The kernel of the differential}

For our results in the case $\OGr_3(n)$, one of the key steps was that the kernel of $d_1$ in the Koszul complex was a free $W_2$-module, cf.~Proposition~\ref{prop:k3-basis}. For $k\geq 4$, it is no longer the case that $\ker(d_1)$ is a free $W_2$-module, but we can compute a resolution for $\ker(d_1)$ as $W_2$-module using the \verb!Macaulay2! code from the appendix, simply by running the line (after specifying $k$ and $n$):
\begin{verbatim}
resolution kernel kosz(k,n).dd_1
\end{verbatim}

Again, we can record a couple of observations:
\begin{itemize}
\item Computing this for a number of examples with $k=4,5,6$, suggests that in general $\ker(d_1)$ has a free resolution of length $k-2$. 
\item The ranks of the free modules in the resolution seem to be fairly complicated for $k>4$, but for $k=4$, most of the time, the resolution has the form $0\to W_2^{\oplus n}\to W_2^{\oplus(n+3)}\to\ker(d_1)\to 0$ for $n=1,2,3$.
\end{itemize}

It seems conceivable that the techniques of Section~\ref{sec:K-presentation} could possibly be adapted to provide a general formula for the resolution of $\ker(d_1)$ in the Koszul complex, as a $W_2$-module. To get an idea of what we could possibly expect, we compute the degrees of generators for $\ker(d_1)$ using \verb!presentation kernel kosz(4,n).dd_1!, for $n=17,\dots,32$:

\begin{center}
\begin{tabular}{|l|l||l|l|}
  \hline
  $n$ & degrees & $n$ & degrees\\
  \hline
  17 & 17, 29, 30, 31  &   25 & 29, 41, 46, 47, 49 \\
  18 & 21, 29, 31, 33, 34 &   26 & 29, 45, 47, 49, 50 \\
  19 & 25, 29, 33, 34, 35 &   27 & 29, 49, 49, 50, 51 \\
  20 & 29, 29, 33, 37, 38, 39 &   28 & 29, 49, 53, 54, 55\\
  21 & 29, 33, 37, 37, 38, 39 &   29 & 29, 53, 54, 55 \\
  22 & 29, 37, 37, 39, 41, 42 & 30 & 29, 55, 57, 57\\
  23 & 29, 37, 41, 42, 43 & 31 & 29, 58, 59, 61\\
  24 & 29, 37, 45, 46, 47 & 32 & 29, 61, 62, 63\\
  \hline
  \end{tabular}
\end{center}

Compared to the previous table, here we consider $\ker(d_1)$ instead of $K$. We also record here the degrees of elements in the Koszul complex, where in the previous subsection, we considered the degrees of generators of $K$. These two things differ by a shift of 1, i.e., the degrees for $K$ are the Koszul degrees minus 1, see the discussion in Section~\ref{sec:koszul-homology}.

By Corollary~\ref{cor:descending_4}, the descended relation for this stretch has degree 29, and that is prominently visible. By Corollary~\ref{cor:ascending-k4}, the ascended relation for given $n$ would have degree $(d-3)+4i$ for $d=16$ and $i=n-d$. The sequence starts with degree 17 for $n=17$, degree 21 for $n=18$, and so on. The stretch, however, ends prematurely at $n=25$, as there is no degree 53 generator for $\ker(d_1)$ in the case $n=26$.

However, another phenomenon is observable. The last three degrees of generators of $\ker(d_1)$ appearing in the above table are always degrees appearing in degree 2 of the Koszul complex, making it likely that these are coming from the image of $d_2$. But then there are some additional degrees not appearing from ascended/descended relations or the Koszul complex: degree 33 for $n=20$, one of the degree 37 relations for $n=21,\dots,24$, degree 41 for $n=25$, degree 45 for $n=26$, degree 49 for $n=27,28$. Currently there is no explanation for the appearance of these relations. 

\subsection{The presentation of $K$}

After discussing the presentation or the resolution of $\ker(d_1)$ in the Koszul complex, we now come to the presentation of $K$ as a $C$-module for the case $k=4$. The following table collects the degrees of the generators of $K$ for $\OGr_4(n)$ with $n=17,\dots,28$. 

\begin{center}
\begin{tabular}{|l|l||l|l||l|l|}
  \hline
  $n$ & degrees & $n$ & degrees & $n$& degrees \\
  \hline
  17 & 17  &  23 & 29, 37 &29 & 29\\
  18 & 21, 29 & 24 & 29, 37 & 30 & 29 \\
  19 & 25, 29 &   25 & 29, 41 & 31 & 29\\ 
  20 & 29, 29, 33 &  26 & 29, 45 & 32 & 29\\  
  21 & 29, 33, 37 &   27 & 29, 49 & & \\
  22 & 29, 37, 37 & 28 & 29, 49 & &\\
  \hline
  \end{tabular}
\end{center}

Again, we can make some observations:
\begin{itemize}
\item The generators whose degree we identified as appearing in the degree 2 part of the Koszul complex have vanished. This supports the previously formulated assumption that these generators of $\ker(d_1)$ are also already in the image of $d_2$.
\item Contrary to the case $k=3$, not all relations are ascended or descended from the relations $q_{2^t-3}=0$. There is an ascending relation between $n=17\stb  22$ and the descending relation persists between $n=18\stb 32$. There is a new relation appearing between $n=20\stb 28$, which also follows an ascending pattern, but this pattern is broken; sometimes the relation stays in the same degree, possibly due to the vanishing of some coefficients (to compare with $k=3$, cf.\ Remark \ref{rmk:nonvanishing_r_coefficients}).
\end{itemize}

We also include one more experiment concerning the degrees of the relations in the presentation of $K$. We check \verb!degrees relations prune K! where \verb!K! has been constructed as anomalous module from the Koszul complex. The following table collects the degrees of relations in the presentation for $K$, for $\OGr_4(n)$ with $n=17,\dots,28$. 

\begin{center}
\begin{tabular}{|l|l||l|l|}
  \hline
  $n$ & degrees & $n$ & degrees\\
  \hline
  17 & ------  &  23 & 43, 44, 45 \\
  18 & 32, 33, 35 & 24 & 43, 44, 45 \\
  19 & 35, 36, 37 &   25 & 45, 47, 48 \\ 
  20 & 35, 36, 37, 41, 45, 45 &  26 & 48, 49, 51\\  
  21 & 39, 40, 41, 41, 45, 49 &   27 & 51, 52, 53 \\
  22 & 40, 41, 43, 45, 45, 53 & 28 & 51, 52, 53\\
  \hline
  \end{tabular}
\end{center}

To state the observation here, we note that the Koszul complex for $k=4$ always has the form
\[
0\to W_2\to W_2^{\oplus 4}\to W_2^{\oplus 6}\xrightarrow{d_2} W_2^{\oplus 4}\xrightarrow{d_1} W_2.
\]
The middle degree always has six generators. The observation we record here is that three of those appeared in the presentation of $\ker(d_1)$ before, and the other three appear now in the above table, as the first three degrees.

As a final observation, checking \verb!resolution prune HH_1 kosz(4,n)! suggests that it is possible to get a fairly reasonably-looking and small free resolution of $K$ as a $W_2$-module by modifying the Koszul complex with the short free resolution of $\ker(d_1)$.

The study of the presentation of $K$ as $C$-module for $\OGr_4(n)$ will be subject of future study.

\subsection{Nontrivial Ext groups}

We now come to one of the bigger problems when going beyond the $k=3$ case. What simplified the presentation of ${\rm H}^*(\OGr_3(n),\mathbb{F}_2)$ as a $C$-module significantly, was the triviality of the Ext-group as discussed in Proposition~\ref{prop:ext-vanishing-k3}. This is no longer true for $k\geq 4$, and we will discuss a couple of examples where the Ext-group is nontrivial below, making again use of the \verb!Macaulay2!-code from the appendix.

Above, we made an observation about getting a resolution of $K$ as a $W_2$-module from a modification of the Koszul complex by the free resolution of $\ker(d_1)$. This could be a more conceptual way to determine the Ext-group $\Ext_C^1(K,C)$ for $\OGr_4(n)$ generally. But even with that, it remains a significant challenge to actually determine the class of the extension
\[
\xymatrix@R-2pc{
0\ar[r]& C\ar[r]&{\rm H}^*(\OGr_4(n),\mathbb{F}_2)\ar[r]& K\ar[r]& 0
}
\]
as an element in $\Ext_C^1(K,C)$ (or even the triviality or nontriviality of this element). We also hope to return to this question in future investigations. Lifts of generators and relations of $K$ to integral cohomology could potentially prove helpful.


In the following, we list the examples of non-trivial Ext-groups for $k=4,5,6$ and small $n$. Here, ``small'' is essentially determined by patience vs running time of the \verb!Macaulay2!-computation. More computational power or patience can easily extend the list.
\begin{itemize}
  \item For $k=4$ and $n\leq 36$, the following examples have non-trivial Ext-group: $n=18$ of rank 1, $n=24$ of rank 2, $n=25$ of rank 1, $n=34$ of rank 1, $n=35$ of rank 2, $n=36$ of rank 4.
  \item For $k=5$ and $n\leq 21$, the following examples have nontrivial Ext-group: $n=13$ of rank 1, $n=14$ of rank 1, $n=15$ of  rank 1, $n=17$ of rank 3, and $n=18$ of rank 5.
  \item For $k=6$, and $n\leq 18$, the following examples have nontrivial Ext-group: $n=12$ of rank 1, $n=13$ of rank 1, $n=14$ of rank 1, $n=15$ of rank 1, $n=16$ of rank 1, $n=17$ of rank 10, and $n=18$ of rank 26.
\end{itemize}

For examples that fall within the scope $(k,n)$ listed above but have trivial Ext-group, i.e., are not mentioned in the above list, a presentation of ${\rm H}^*(\OGr_k(n),\mathbb{F}_2)$ can be obtained using our \verb!Macaulay2!-code: simply get a presentation of $K$ as $C$-module, and then ${\rm H}^*(\OGr_k(n),\mathbb{F}_2)\cong C\oplus K$ as $C$-module. Note, however, that even if the Ext-group is non-trivial, it could still be possible that the cohomology of $\OGr_k(n)$ splits as $C\oplus K$; for now, we cannot make any more definite statements on nontriviality of the Ext-class. 

To conclude, we discuss what the Ext-group calculation by hand could look like in two specific cases:

\begin{example}
We consider the case $\OGr_4(18)$ in which there is a nontrivial Ext-group. The module $K=\ker w_1$ is generated by two elements $a_{20}$ and $a_{28}$ with three relations
\[
w_2^4w_3a_{20}+w_3a_{28}, \qquad (w_2^6+w_3^4+w_2^4w_4)a_{20}+w_4a_{28}, \qquad w_2^5w_4a_{20}+(w_3^2+w_2w_4)a_{28}
\]
There are six relations between relations in degrees 35, 37, 38, 43, 45 and 46, where $C^{35}=0$ and all of the remaining degrees are above the top degree of $C$. In particular, the Ext-group is the quotient of the differential $d_0\colon C^{20}\times C^{28}\to C^{31}\times C^{32}\times C^{34}$ given by the above relations. Note that the $\mathbb{F}_2$-dimensions of $C^{31}$ and $C^{34}$ are one and the $\mathbb{F}_2$-dimension of $C^{32}$ is two, i.e., the target of $d_0$ has $\mathbb{F}_2$-dimension~$4$. Using the \verb!diffrank!-function from the appendix, we can compute that the differential has rank 3, with a basis for the image given by
\begin{align*}
  d_0(0,w_2^3w_3^6w_4)&=(0,w_2w_3^2w_4^6,0),\\
  d_0(0,w_2^3w_3^2w_4^4)&=(0,w_2^2w_4^7,w_3^2w_4^7), \\
  d_0(0,w_4^7)&=(w_3w_4^7,0,w_3^2w_4^7).
\end{align*}
 We see that by adding suitable degree 28 elements, we can always achieve that the relations in degrees 31 and 32 are trivial. In such a normal form, the extension is determined by the relation in degree 34.  For example, one representative of the non-trivial extension of $K$ by $C$ in this case is the $C$-module generated by $\alpha_0,\alpha_{20}$ and $\alpha_{28}$ subject to the relations
\[
w_2^4w_3\alpha_{20}+w_3\alpha_{28}=(w_2^6+w_3^4+w_2^4w_4)\alpha_{20}+w_4\alpha_{28}=0, w_2^5w_4\alpha_{20}+(w_3^2+w_2w_4)\alpha_{28}=w_3^2w_4^7\alpha_0.
\]
\end{example}

\begin{example}
  We consider the case $\OGr_6(12)$ with a non-trivial Ext-group. The module $K=\ker w_1$ is generated by three elements $a_{14}$, $a_{15}$ and $a_{16}$ with six relations
\[
w_3a_{14}, \quad w_5a_{14}+w_3a_{16}, \quad w_2^2a_{15}+w_3a_{16},
\]
\[ (w_2w_4+w_6)a_{14}+(w_2w_3+w_5)a_{15}+w_4a_{16}, \quad w_2^2a_{16}, \quad w_3^2a_{15}+(w_2w_3+w_5)a_{16}.
\]
There are 13 relations between relations in degrees between 21 and 27. Of these, only the two relations
\[
w_2^2R_{17}, \qquad w_5R_{17}+w_3(R_{19,1}+R_{19,2})
\]
in degree 21 and 22 matter, all the others are above the top degree of $C$ which is 22. The differential $d_1$ has rank 1, and the differential $d_0$ has rank 7, and the dimension of the space of 1-cocycles is 9. The Ext-group therefore is a 1-dimensional $\mathbb{F}_2$-vector space.
\end{example}

\appendix
\section{Computing Koszul resolutions, presentations and Ext-groups with Macaulay2}

In the appendix, we include some \verb!Macaulay2! code that allows to compute Koszul complexes and presentations for the anomalous module $K$ as well as the relevant Ext-groups for the oriented Grassmannian computations. We briefly indicate what the code is doing, how it is used and how to compute one of the examples discussed in Section~\ref{sec:beyondk3}.

We start off with a couple of lines containing the preliminary definitions of Giambelli determinants. The function \verb!giambrow! produces the rows of the matrix $Q_j$ in \eqref{eq:giambelli} starting with a number of zeroes \verb!zs! and the list of Stiefel--Whitney classes \verb!ws!. Then \verb!giambmx! puts together the matrix, and \verb!giambdet! returns the relevant Giambelli determinant.

\begin{verbatim}
zs = (j) -> (return for i from 1 to j list 0)
ws = (j) -> (return {1, 0} | for i from 2 to j list w_i)

giambrow = (d) -> (return zs(d-1) | ws(d))

giambmx = (d) -> (
return for i from 1 to d list take(giambrow(d), {d - i + 1, 2 * d - i})
)

giambdet = (d) -> (return determinant(matrix(giambmx(d))))
\end{verbatim}

Next, we have a couple lines to construct the Koszul complex for the ideal $(q_{n-k+1},\dots,q_n)$ in $W_2$. The function \verb!q(k,j)! encodes the recursive definition of $q_j$ in $W_2=\mathbb{F}_2[w_2,\dots,w_k]$ from \eqref{eq:recursion}, and then \verb!kosz! constructs the Koszul complex (after turning the $q_i$ into a matrix to be used by the \verb!Macaulay2! function constructing the Koszul complex). Note that \verb!kosz(k,l)! constructs the Koszul complex relevant for $\OGr(k,k+l)$.

\begin{verbatim}
q = (k, j) -> (
  if j == 0 then return 1;
  if j <= k then return giambdet(j);
  return sum(2..k, i -> q(k, j-i) * w_(i)))

kosz = (k, l) -> (
  R = GF(2)[w_2..w_k, Degrees => {2..k}];
  f=matrix{for i from l+1 to k+l list q(k, i)};
  C = koszul f;
  return C)
\end{verbatim}

After this, the Koszul homology can be accessed. As we discussed in Section~\ref{sec:koszul-homology}, the characteristic subring and the anomalous module appear as 0th and 1st Koszul homology, respectively. Note, however, that \verb!HH_i cplx! would return the $i$-th homology of the complex as a $W_2$-module. The following lines then do some conversion: \verb!charsubring! takes the presentation of \verb!HH_0 cplx! as a $W_2$-module and uses it to actually construct $C$ as a quotient ring of $W_2$. The function \verb!anomalous! converts the description of \verb!HH_1 cplx! as $W_2$-module to a description as $C$-module. Using \verb!prune! everywhere helps cut down the complexity of the resulting presentations to human-readable size and form.

\begin{verbatim}
charsubring = (cplx) -> (
    C = prune HH_0 cplx;
    I = ideal (flatten entries presentation C);
    return R/I)

anomalous = (cplx) -> (
    K = prune HH_1 cplx;
    return cokernel (C**presentation(K)))
\end{verbatim}

Having $C$ and $K$, it is now easy to compute the Ext-group (or its rank) as follows. It is important to note that \verb!Ext^1! would compute a graded Ext-group for $K$ and $C$ as graded modules over the graded ring $C$ (with the grading coming from the grading of $W_2$ where $w_i$ has degree $i$). As explained in Section~\ref{sec:koszul-homology}, $K$ is actually the shift of the first homology of the Koszul complex by one. The command \verb!basis(-1,Ext^1(trim K,C))! takes that into account, actually computing a basis of the Ext-group that classifies degree 0 extensions of $C$ by a shifted copy of $K$.

\begin{verbatim}
rankExt = (k,l) -> (
  cplx = kosz(k,l);
  C = charsubring(cplx);
  K = anomalous(cplx);
  return #(transpose(entries basis(-1,Ext^1(trim K,C)))))
\end{verbatim}

Finally, some mystery code to compute information pertaining to the rank of the differential of the Koszul complex in a specified degree:

\begin{verbatim}
diffrank = (cplx, deg) -> (
  diffmat = cplx.dd_deg;
  baselist = apply(fold((a,b)->a**b, apply(degrees cplx#deg, 
    x->{0}|(flatten entries basis(-x-{1},C)))), deepSplice @@ toList);
  basemat = transpose(matrix(baselist));
  return basis(image(diffmat*basemat)))    
\end{verbatim}

\section{Basics on Ext-groups}
\label{sec:ext-basics}

In the following section, we recall the standard basics on Ext-groups and how they relate to extensions, such as they can be found e.g. in \cite{weibel}. This is to support our discussion in Section~\ref{sec:ext-class} (where we show triviality of Ext-groups for $k=3$) and Section~\ref{sec:beyondk3} (where we exhibit examples of nontrivial Ext-groups in some $k>3$ cases). We point out that the discussion below is for rings and modules, while our application to the oriented Grassmannians is actually about \emph{graded} modules over the \emph{graded} ring $C$. 

Fix a ring $R$. Recall that for two $R$-modules $A$ and $B$, the Ext-groups ${\rm Ext}^i_R(A,B)={\rm RHom}^i_R(A,B)$ can be computed as the cohomology groups of the complex ${\rm Hom}_R(P_\bullet,B)$ where $P_\bullet\to A$ is a chosen projective resolution (in the category of $R$-modules).

\begin{remark}  
	For an $R$-module $K$, the start of a projective resolution looks like
	\[\cdots\to P_2\to P_1\to P_0\to K\to 0.\]
	We can choose $P_0$ to be the free $R$-module on a chosen set of $R$-generators of $K$, with $P_0\to K$ mapping the elements corresponding to the generators to the respective generators in $K$. Then we can choose $P_1$ to be the free $R$-module on a chosen set of $R$-relations between the generators, with the morphism $P_1\to P_0$ mapping the generator corresponding to a relation in $K$ to ``itself'', written out in terms of the generators of $P_0$ (which correspond to the generators of $K$). Then we choose $P_2$ to be the free $R$-module on relations between relations, and so on. Given the resolution, we can compute ${\rm Ext}^1_R(K,C)$ as the first cohomology of the complex
	\[0\to {\rm Hom}_R(P_0,C)\to {\rm Hom}_R(P_1,C)\to {\rm Hom}_R(P_2,C)\to \cdots\]
	The differentials for this complex are induced by composition $P_{i+1}\to P_i\to C$.
	
	In terms of the above description of a resolution in terms of generators and (higher) relations of $K$, a class in ${\rm Ext}^1(K,C)$ is described by a choice of elements $r_1,\dots,r_m$ of $C$ corresponding to the relations in $K$. The cycle condition translates into the requirement that these elements satisfy the ``relations between relations'' in $K$. The Ext-class is a coboundary if there is a choice of elements $g_1,\dots,g_n$ of $C$ corresponding to the generators of $K$, such that $r_1,\dots,r_m$ are actually the relations of $K$ written in terms of the $g_1,\dots,g_n$.
\end{remark}

We also briefly recall that ${\rm Ext}^1(A,B)$ classifies equivalence classes of extensions
\[0\to B\to E\to A\to 0\]
of $R$-modules. The class $[E]\in {\rm Ext}^1(A,B)$ associated to the extension is given as the image of ${\rm id}_A$ under the boundary map $\partial\colon {\rm Hom}_R(A,A)\to {\rm Ext}^1(A,B)$ in the long exact Ext-group sequence associated to the given extension.

\begin{remark}
	The boundary map can be computed as usual in the exact sequence of complexes
	\[
	0\to {\rm Hom}_R(P_\bullet,B)\to {\rm Hom}_R(P_\bullet,E)\to {\rm Hom}_R(P_\bullet,A)\to 0.
	\]
	The class of ${\rm id}_A$ in ${\rm Hom}_R(A,A)={\rm Ext}^0_R(A,A)$ is represented by the map $P_0\to A$ in the projective resolution of $A$. Lift this to ${\rm Hom}_R(P_0,E)$ by lifting the generators of $A$ to $E$. Apply the boundary map ${\rm Hom}_R(P_0,E)\to {\rm Hom}_R(P_1,E)$ by writing out the relations in the presentation of $A$ in terms of the lifts of generators to $E$. By construction, the composition $P_1\to E\to A$ will be 0, so that the morphism $P_1\to E$ is in the image of ${\rm Hom}_R(P_1,B)\to {\rm Hom}_R(P_1,E)$. The resulting 1-cochain represents $[E]\in {\rm Ext}^1(A,B)$. In particular, the Ext-class can be computed explicitly from the $R$-module structure of $E$ and an $R$-module presentation of $A$.
\end{remark}

From the class $[E]\in {\rm Ext}^1(A,B)$, we can indeed completely recover the extension of $A$ by $B$. Starting from the extension $0\to M\to P_0\to A\to 0$ arising from a projective resolution of $A$, we obtain an exact sequence
\[
{\rm Hom}_R(P_0,B)\to {\rm Hom}_R(M,B)\to {\rm Ext}^1_R(A,B)\to 0
\]
by applying ${\rm Ext}^\bullet_R(-,B)$ (and noting that $P_0$ is projective). We choose a preimage $e\colon M\to B$ of $[E]\in {\rm Ext}^1_R(A,B)$ and obtain an extension
\[0\to B\to P_0\cup_M B\to A\to 0
\]
by pushout of $M\to P_0$ along $e\colon M\to B$.

\begin{remark}
	In the concrete situation of the cohomology rings of oriented Grassmannians, we can get an explicit presentation for the cohomology ring $H\cong {\rm H}^*(\widetilde{\rm Gr}(k,n),\mathbb{F}_2)$ as a $C$-module from the Ext-class $[H]\in {\rm Ext}^1_C(K,C)$. We again take the free resolutions starting with a presentation of $K$ as $C$-module. This provides an injection $P_1/P_2=M\hookrightarrow P_0$, where $M$ is the submodule of relations in $P_0$. The Ext-class $[H]$ is represented by a 1-cocycle $e\colon P_1\to C$, where the cocycle property means that the 1-cocycle factors through a morphism $P_1/P_2\to C$. Then $H\cong P_0\cup_{P_1/P_2} C$. Spelling this out, we can present $H$ as a $C$-module generated by 1 and the generators $g_1,\dots,g_m$ of $K$, and for each relation $R$ in $K$, we have a relation identifying the relation spelled out in terms $g_1,\dots,g_m$ with the image of the corresponding relation element under the map $e\colon P_1\to C$.
\end{remark}

\end{document}